\documentclass[10pt]{amsart}
\usepackage[margin=1.5in]{geometry}
\usepackage{color}
\newtheorem{theorem}{Theorem}[section]
\newtheorem{lemma}[theorem]{Lemma}
\newtheorem{proposition}{Proposition}[section]
\newtheorem{corollary}{Corollary}[section]
\theoremstyle{definition}

\theoremstyle{remark}

\numberwithin{equation}{section}

\begin{document}

\title{Ellipsoidal BGK model near a global Maxwellian }

\author{Seok-Bae Yun}
\address{Department of Mathematics, Sungkyunkwan University, Suwon 440-746, Republic of Korea}
\email{sbyun01@skku.edu}




\begin{abstract}
The BGK model has been widely used in place of the Boltzmann equation because of the qualitatively satisfactory results
it provides at relatively low computational cost.
There is, however, a major drawback to the BGK model: The hydrodynamic limit at the Navier-Stokes level is not correct.
One evidence is that the Prandtl number
computed using the BGK model does not agree with what is derived from the Boltzmann equation.
To overcome this problem, Holway \cite{Holway} introduced the ellipsoidal BGK model where the local Maxwellian is replaced by a non-isotropic
Gaussian. In this paper, we prove the existence of classical solutions of the ES-BGK model when the initial data is a
small perturbation of the global Maxwellian. The key observation is that the degeneracy of the ellipsoidal BGK model is comparable to that of the original BGK model or the Boltzmann equation in the range $-1/2<\nu<1$.
\end{abstract}
\maketitle
\tableofcontents
\section{Introduction}The dynamics of a non-ionized monatomic rarefied gas is governed
by the Boltzmann equation. But the complex structure of the collision operator has long been a major obstacle for
theoretical and computational investigation of the Boltzmann equation. To overcome this difficulty, Bhatnagar et al.\cite{B-G-K},
and independently Walender \cite{Wal},
introduced a model equation called the BGK model, where the collision operator is replaced by a relaxation operator.
Since then, it has been widely used in place of the Boltzmann equation for various computational experiments,
since this model provides very satisfactory results at relatively low computational cost compared to the Boltzmann equation.
But the BGK model has a major drawback. Hydrodynamic limit at the Navier Stokes level is not satisfactory in that
the Prandtl number - defined as the ratio between the viscovity  and the thermal conductivity - computed using the BGK model is incorrect: The Prandlt number for the Navier-Stokes equation is around 0.7, but
the  computation using the BGK model yields 1.
To resolve this problem, Holway suggested a variant of the BGK model, called the ellipsoidal BGK model (ES-BGK model) \cite{Holway}:
\begin{eqnarray}\label{ES-BGK}
\begin{split}
&\partial_t F+v\cdot\nabla_x F=A_{\nu}\big(\mathcal{M}_{\nu}(F)-F\big),\cr
&\hspace{0.84cm} F(x,v,0)=F_0(x,v).
\end{split}
\end{eqnarray}
$F(x,v,t)$ denotes the velocity distribution function representing the number density on the phase space point $(x,v)$ in
$\mathbb{T}^3_x\times \mathbb{R}^3_v$ at time $t\in \mathbb{R}_+$.  $A_{\nu}$ is the collision frequency whose explicit form will be given later.
The non-isotropic Gaussian $\mathcal{M}_{\nu}(F)$ in the r.h.s of (\ref{ES-BGK}) is defined as follows: First, we define the macroscopic density $\rho$,
bulk velocity $U$, temperature $T$ and the stress tensor $\Theta$ by
\begin{eqnarray*}
\rho(x,t)&=&\int_{\mathbb{R}^3}F(x,v,t)dv,\cr
\rho(x,t)U(x,t)&=&\int_{\mathbb{R}^3}F(x,v,t)vdv,\cr
3\rho(x,t) T(x,t)&=&\int_{\mathbb{R}^3}F(x,v,t)|v-U(x,t)|^2dv,\cr
\rho(x,t)\Theta(x,t)&=&\int_{\mathbb{R}^3}F(x,v,t)(v-U)\otimes(v-U)dv,
\end{eqnarray*}
and introduce the temperature tensor $\mathcal{T}_{\nu}$  as a linear combination of $T$ and $\Theta$:
\begin{eqnarray*}
\mathcal{T}_{\nu}&=&\left(
\begin{array}{ccc}
(1-\nu) T+\nu\Theta_{11}&\nu\Theta_{12}&\nu\Theta_{13}\cr
\nu\Theta_{21}&(1-\nu)T+\nu\Theta_{22}&\nu\Theta_{23}\cr
\nu\Theta_{31}&\nu\Theta_{32}&(1-\nu) T+\nu\Theta_{33}
\end{array}
\right)\cr
&=&(1-\nu) T Id+\nu\Theta.
\end{eqnarray*}
The non-isotropic Gaussian $\mathcal{M}_{\nu}$ is now defined as follows:
\begin{eqnarray*}
\mathcal{M}_{\nu}(F)=\frac{\rho}{\sqrt{\det(2\pi \mathcal{T}_{\nu})}}\exp\left(-\frac{1}{2}(v-U)^{\top}\mathcal{T}^{-1}_{\nu}(v-U)\right).
\end{eqnarray*}
We note that the temperature is recovered as the trace of $\mathcal{T}_{\nu}$:
\begin{eqnarray*}
3T=\Theta_{11}+\Theta_{22}+\Theta_{33}=tr\Theta=tr\mathcal{T}_{\nu}.
\end{eqnarray*}
The collision frequency $A_{\nu}$ takes the following explicit form:
\[
A_{\nu}=\frac{\rho~T}{1-\nu},\quad-\frac{1}{2}<\nu<1.
\]
The free parameter $\nu$ is introduced to derive the correct Prandtl number.
The restriction on the range of $\nu$ is imposed to guarantee that the temperature tensor $\mathcal{T}_{\nu}$ remains positive definite. (See \cite{A-L-P-P}).
Prandtl number computed via the Chapman-Enskog expansion using the ES-BGK model is given by $Pr=1/(1-\nu)$ (See \cite{A-L-P-P,C-C,Holway,Stru}).
The two most important cases in the range $-1/2<\nu<1$ are $\nu=0$ and $\nu=(Pr-1)/Pr\approx-3/7$:
When $\nu=0$, (\ref{ES-BGK}) reduces to the classical BGK model, whereas $\nu=(Pr-1)/Pr$ corresponds to the ES-BGK model with the correct Prandtl number.\newline
\indent The relaxation operator of the ES-BGK model satisfies the following cancelation property \cite{A-B-L-P,A-L-P-P}:
\begin{eqnarray*}
\begin{split}
\int_{\mathbb{R}^3}\big(\mathcal{M}_{\nu}(F)-F\big)
\left(\begin{array}{c}
1\cr v\cr|v|^2
\end{array}\right)
dv=0,
\end{split}
\end{eqnarray*}
which implies the conservation of mass, momentum and energy:
\begin{eqnarray}\label{ConservationLawsF}
\begin{split}
\int_{\mathbb{T}^3_x\times\mathbb{R}^3_v}F(x,v,t)dxdv&=\int_{\mathbb{T}^3_x\times\mathbb{R}^3_v}F_0(x,v)dxdv,\cr
\int_{\mathbb{T}^3_x\times\mathbb{R}^3_v}F(x,v,t)vdxdv&=\int_{\mathbb{T}^3_x\times\mathbb{R}^3_v}F_0(x,v)vdxdv,\cr
\int_{\mathbb{T}^3_x\times\mathbb{R}^3_v}F(x,v,t)|v|^2dxdv&=\int_{\mathbb{T}^3_x\times\mathbb{R}^3_v}F_0(x,v)|v|^2dxdv.
\end{split}
\end{eqnarray}
Entropy dissipation property was established recently in \cite{A-L-P-P}:
\begin{eqnarray*}
\frac{d}{dt}\int_{\mathbb{T}^3_x\times\mathbb{R}^3_v}F(t)\log F(t)dxdv\leq 0.
\end{eqnarray*}
It is important to note that, as in the case of the original BGK model or the Boltzmann equation,
 the only possible equilibrium state for (\ref{ES-BGK}) is the local Maxwellian:
\[
M(F)=\frac{\rho}{(2\pi T)^{3/2}}e^{-\frac{|v-U|^2}{2T}}.
\]
To see this, let's assume that $\mathcal{M}_{\nu}(f)=f$. We then recall the definition of $\Theta$ and $\mathcal{T}_{\nu}$ to see that
\begin{eqnarray*}
\begin{split}
\int_{\mathbb{R}^3}F(v)(v-U)\otimes(v-U)dv&=\rho\Theta,\cr
\int \mathcal{M}_{\nu}(F)(v)(v-U)\otimes(v-U)dv&=\rho\mathcal{T}_{\nu}.
\end{split}
\end{eqnarray*}
Therefore, upon multiplying $(v-U)\otimes(v-U)$ to both sides of $\mathcal{M}_{\nu}(F)=F$ and integrating with respect to $v$, we have
\begin{eqnarray*}
\rho\mathcal{T}_{\nu}=\rho\Theta.
\end{eqnarray*}
In view of the definition of $\mathcal{T}_{\nu}$, this leads to
\begin{eqnarray*}
(1-\nu)TId+\nu\Theta=\Theta.
\end{eqnarray*}
Thus, $\Theta=T Id$, and we see, from the definition of $\mathcal{T}_{\nu}$, that $\mathcal{T}_{\nu}=T Id$ for $3\times 3$ identity matrix $Id$.
This gives
\begin{eqnarray*}
\mathcal{M}_{\nu}(F)&=&\frac{\rho}{\sqrt{\det(2\pi T Id)}}\exp\left(-\frac{1}{2}(v-U)^{\top}\{TId\}^{-1}(v-U)\right)\cr
&=&\frac{\rho}{(2\pi T)^{3/2}}e^{-\frac{|v-U|^2}{2T}}\cr
&=&M(F).
\end{eqnarray*}
That is,  $\mathcal{M}_{\nu}(F)$ reduces to the usual local Maxwellian $M(F)$.\newline\newline
\indent In this paper, we study the existence of classical solutions of (\ref{ES-BGK}) and their asymptotic behavior when the initial data is a small perturbation of the normalized global Maxwellian:
\begin{eqnarray}\label{global maxwellian}
\mu(v)=\frac{1}{\sqrt{(2\pi)^3}}e^{-\frac{|v|^2}{2}}.
\end{eqnarray}
We define the perturbation $f$ around $\mu$ by the relation: $F(x,v,t)=\mu+\sqrt{\mu}f(x,v,t)$ and, accordingly,
$F_0(x,v)=\mu+\sqrt{\mu}f_0(x,v)$.
Then, after linearization around the global Maxwellian, the ES-BGK model takes the following form
(See Section 2 for precise definition of each term):
\begin{eqnarray}\label{LBGK0}
\begin{split}
\partial_t f+v\cdot\nabla_xf&=L_{\nu}f+\Gamma(f),\cr
f(x,v,0)&=f_0(x,v).
\end{split}
\end{eqnarray}
where $L_{\nu}$ denotes the linearized relaxation operator and $\Gamma(f)$ is the nonlinear part.
In section 2, we verify that $L_{\nu}$ can be represented as a $\nu$-perturbation of the linearized relaxation operator
of the original BGK model:
\begin{eqnarray}\label{putting0}
L_{\nu}f=(P_0f-f)+\nu P_{1}f+\nu P_{2}f.
\end{eqnarray}
Here, $P_0$ denotes the macroscopic projection operartor on the the linear space generated by $\{\sqrt{\mu},v\sqrt{\mu},|v|^2\sqrt{\mu}\}$.
$P_1$ and $P_2$ are operators related to the burnett functions, which play a crucial role in the hydrodynamic limit of the Boltzmann equation
at the Navier-Stokes level. (See \cite{B-G-L1}).
In general, the coercivity estimate of the linearized collision or relaxation operators for spatially inhomogeneous collisional kinetic
equations are degenerate,
and the major difficulty in obtaining the global existence in the perturbative regime lies in removing the degeneracy
to recover the full coercivity \cite{Guo-whole,Guo-VMB,Guo-VPB}.
When the spatial variable lies in $\mathbb{T}^3$, the usual recipe is the use of the Poincare inequality  together with a system of macroscopic equations and the conservation laws (See, for example, \cite{Guo-VMB,Guo-VPB}). In the whole space, where the
Poincare inequality is not available, additional consideration has to be made to compensate the still lingering degeneracy \cite{Duan,Duan-StrainV,Duan-STrainM,Kawashima,Ukai,U-T}.
Therefore, it is very important to capture the degenerate coercivity estimate of the linearized relaxation operator first.
In our case, it is not clear whether the presence of the additional terms $P_1$ and $P_2$ make the linearized relxation operator more degenerate or not.
In Theorem \ref{degenerate coercivity}, we show that, for $-1/2<\nu<1$, the
degenerate coercive estimate of $L_{\nu}f$ is comparable to that of $L_0f=(P_0-I)f$, for which the usual energy method is well-established
(See Theorem \ref{degenerate coercivity}):
\begin{eqnarray*}
\langle L_{\nu}f,f\rangle_{L^2_{v}}\leq-C_{\nu}\|\{I-P_0\}f\|^2_{L^2_{v}},    \quad(-1/2<\nu<1),
\end{eqnarray*}
for some constant $C_{\nu}>0$.
This indicates that the dissipative property of the linearized relaxation operator for the ES-BGK model is essentially same as that of the
BGK model or Boltzmann equation.
\newline
\indent On the other hand, since the ES-BGK model is obtained by replacing the temperature function $T$ by the temperature tensor $\mathcal{T}_{\nu}$
in the classical BGK model, additional difficulties related to $\mathcal{T}_{\nu}$, which was not observed in the classical BGK model arise.
First, in each step of the iteration scheme designed to obtain the local in time existence of the solution, we need to check that the temperature tensor
remains strictly positive definite, which is established in Proposition \ref{Equivalence T} as:
\begin{eqnarray*}
\frac{C^{-1}_{\nu2}}{T(x,t)}Id\leq\mathcal{T}^{-1}_{\nu}(x,t)\leq \frac{C^{-1}_{\nu1}}{T(x,t)}Id.
\end{eqnarray*}
where $C_{\nu1}=\min\{1-\nu, 1+2\nu\}$ and $C_{\nu1}=\max\{1-\nu, 1+2\nu\}$.
This also shows why the restriction of the range of the free parameter in the interval $(-1/2,1)$ is crucial: It is only
in this range that the temperature tensor is comparable to $T$, and, therefore, the non-isotropic Gaussian is comparable to the local Maxwellian.
Secondly, due to the presence of the free parameter $\nu$ in the definition of the temperature field $\mathcal{T}_{\nu}$,
it is a priori not clear whether the nonlinear perturbation $\Gamma(f)$ can be estimated uniformly with respect to $\nu$ near $\nu=0$
because the the inverse of the temperature tensor $\mathcal{T}^{-1}_{\nu}$ may have problematic terms involving $1/\nu$.
Such a singularity at $\nu=0$ is undesirable considering that the case $\nu=0$ corresponds to the classical BGK model.
The above equivalence estimate guarantees that such singularity never shows up when $-1/2<\nu<1$.
\newline\newline
\indent The mathematical theory for the BGK model has a rather short history.
The first rigorous existence result can be traced back to Ukai \cite{Ukai-BGK}, where he considered stationary problem for 1 dimensional BGK model in a periodic bounded domain. Perthame established the existence of weak solutions
of the BGK model with constant collision frequency in \cite{Perthame} assuming only the finite mass, momentum, energy and entropy.
See also \cite{B-P}.
The uniqueness was considered in a more stringent functional space involving the pointwise decay in velocity \cite{P-P}.
Mischler considered similar problems in the whole space in \cite{Mischler}. Extension to $L^p$ was carried out in \cite{Z-H}.
Issautier established regularity estimates for the BGK model and proved the convergence of a Monte-Carlo type scheme
to the regular distribution function in \cite{Issau}. The convergence property of a semi-Lagrangian scheme for the BGK model was
studied in \cite{R-S-Y}.
In near Maxwellian regime, Bellouquid \cite{Bello} obtained the global well posedness in the whole space using
Ukai's spectral analysis argument \cite{Ukai}. In the periodic case, Chan employed the energy method developed by Liu et al. \cite{L-Y-Y} to establish the global in time classical solution near global Maxwellians \cite{Chan}. The convergence rate to the equilibrium was not known in this work, which was derived in \cite{Yun}.
For fluid dynamic limit of the BGK model, see \cite{SR1,SR2}.
The ES-BGK model has attracted only limited attention until very recently since it was not clear whether
the entropy dissipation property holds for this model.
It was proved in the affirmative, at least at the formal level, in \cite{A-L-P-P}, which revived the interest
on this model. To our knowledge, no existence result has been established for the ellipsoidal BGK. For numerical test
for the ES-BGK model, we refer to \cite{A-B-L-P,F-J,GT,M-S,Z-Stru}. 
For general review of the mathematical and physical theory of the Boltzmann equation and the BGK model, see \cite{C,C-I-P,GL,Sone,Sone2,Stru-book,V}.
\newline\newline
Before proceeding further, we define some notations.
\begin{itemize}
\item When there is no risk of confusion, we use generic constants $C$. Their value may change from line to line but does not depend on important
parameters.
\item We define the index set $i<j$ by
\[
\sum_{i<j}a_{ij}=a_{12}+a_{23}+a_{31}.
\]
\item $e_i$ $(i=1,2,3)$ denote the standard coordinate unit vectors in $\mathbb{R}^3$.
\item $0^n$ denotes $n$-dimensional zero vector.
\item $I(m,n;a,b)$ denotes a $(m+n)\times(m+n)$ diagonal matrix whose first $m$ diagonal elements are $a$ and following $n$ diagonal elements are $b$.
\item  $\langle\cdot,\cdot\rangle_{L^2_{v}}$ and $\langle\cdot,\cdot\rangle_{L^2_{x,v}}$ denote the standard $L^2$ inner product on
$\mathbb{R}^3_v$ and $\mathbb{T}^3_x\times \mathbb{R}^3_v$ respectively:
\begin{eqnarray*}\begin{split}
\langle f,g\rangle_{L^2_{v}}&=\int_{\mathbb{R}^3}f(v)g(v)dv,\cr
\langle f,g\rangle_{L^2_{x,v}}&=\int_{\mathbb{T}^3\times \mathbb{R}^3}f(x,v)g(x,v)dxdv.
\end{split}\end{eqnarray*}
\item  $\|\cdot\|_{L^2_{v}}$ and $\|\cdot\|_{L^2_{x,v}}$ denote the standard $L^2$ inner norms on
$\mathbb{R}^3_v$ and $\mathbb{T}^3_x\times \mathbb{R}^3_v$ respectively:
\begin{eqnarray*}
&&\|f\|^2_{L^2_{v}}=\Big(\int_{\mathbb{R}^3}|f(v)|^2dv\Big)^{\frac{1}{2}},\cr
&&\|f\|^2_{L^2_{x,v}}=\Big(\int_{\mathbb{T}^3\times \mathbb{R}^3}|f(x,v)|^2dxdv\Big)^{\frac{1}{2}}.
\end{eqnarray*}
\item We employ the following notations for the multi-indices and differential operators:
\begin{eqnarray*}
\alpha=[\alpha_0,\alpha_1,\alpha_2,\alpha_3],\quad \beta=[\beta_1,\beta_2,\beta_3],
\end{eqnarray*}
and
\begin{eqnarray*}
\partial^{\alpha}_{\beta}&=&\partial^{\alpha_0}_t\partial^{\alpha_1}_{x_1}\partial^{\alpha_2}_{x_2}\partial^{\alpha_3}_{x_3}
\partial^{\beta_1}_{v_1}\partial^{\beta_2}_{v_2}\partial^{\beta_3}_{v_3}.\cr
\end{eqnarray*}
For simplicity, when only the spatial derivatives are involved, we write $\partial^{\alpha}_x=\partial^{\alpha_1}_{x_1}\partial^{\alpha_2}_{x_2}\partial^{\alpha_3}_{x_3}$.
\end{itemize}
\subsection{Main results}
We now state our main result. We first define the high order energy
functional $\mathcal{E}(f(t))$:
\begin{eqnarray*}
\mathcal{E}\big(f(t)\big)
=\frac{1}{2}\sum_{|\alpha|+|\beta|\leq N}\|\partial^{\alpha}_{\beta}f(t)\|^2_{L^2_{x,v}}+\sum_{|\alpha|+|\beta|\leq N}\int^t_0\|\partial^{\alpha}_{\beta}f(s)\|^2_{L^2_{x,v}}ds.
\end{eqnarray*}
%
%
%
%
\begin{theorem}
 Let $-1/2<\nu<1$ and $N\geq 4$. Let $F_0=\mu+\sqrt{\mu}f_0\geq 0$ and suppose $f_0$ satisfies (\ref{ConservationLawsf}).
 Then there exist positive constants $\delta_{\nu}$  and $C=C(N,\nu)$,
 such that if $\mathcal{E}(f_0)<\delta_{\nu}$, then there exists a unique global solution $f$ to (\ref{LBGK0}) such that
 \begin{enumerate}
 \item The distribution function is non-negative for all $t\geq 0$:
 \[
 F=\mu+\sqrt{\mu}f\geq 0,
 \]
 and satisfies the conservation laws (\ref{ConservationLawsf}).
 \item The high order energy functional $\mathcal{E}\big(f(t)\big)$ is uniformly bounded:
 \[
 \mathcal{E}\big(f(t)\big)\leq C\mathcal{E}\big(f_0\big).
 \]
\item The distribution function converges to the global equilibrium exponentially fast:
\[
\sum_{|\alpha|+|\beta|\leq N}\|\partial^{\alpha}_{\beta}f(t)\|_{L^2_{x,v}}\leq C e^{-C^{\prime}t}
\]
for some constant $C$ and $C^{\prime}$.
\item If $\bar{f}$ denotes another solution corresponding to initial date $\bar{f}_0$ satisfying the same assumptions, then we have the following
uniform $L^2$-stability estimate:
\[
\|f(t)-\bar f(t)\|_{L^2_{x,v}}\leq C\|f_0-\bar{f}_0\|_{L^2_{x,v}}.
\]
\end{enumerate}
\end{theorem}
This paper is organized as follows: In section 2, we consider the derivation of the linearized ES-BGK equation and the main result is stated. We also derive the coercive estimate and determine the kernel of $L_{\nu}$.
In section 3, various estimates on the macroscopic field are established and, based on this, the local in time existence is obtained.
In section 4, the nonlinear energy estimate is derived, which readily leads to the global existence and the asymptotic behavior.
\section{Linearization}
In this section, we consider the linearzation of the ES-BGK model around the global Maxwellian (\ref{global maxwellian}).
%
%
%
%
For some technical reason, we define $G_{\nu}$ as follows:
\begin{eqnarray*}
G_{\nu}=\frac{1-\nu}{3}\left\{\frac{3\rho T+\rho |U|^2}{2}\right\}Id+\nu\left(\frac{\rho \Theta+\rho U\otimes U}{2}\right)
-\frac{\rho}{2}Id.
\end{eqnarray*}
Due to the symmetry of $G$, we can view $G$ as an element in $\mathbb{R}^6$:
\begin{eqnarray*}
\left\{G_{11},G_{22},G_{33},G_{12},G_{23},G_{31}\right\}.
\end{eqnarray*}
We also define $\mathcal{J}_{\nu}$ to be the Jacobian matrix for the change of variable $(\rho, U, \mathcal{T}_{\nu})\rightarrow (\rho, \rho U, G)$:
\begin{eqnarray*}
\mathcal{J}_{\nu}\equiv\frac{\partial(\rho, \rho U, G_{\nu})}{\partial(\rho, U, \mathcal{T}_{\nu})}.
\end{eqnarray*}
\begin{lemma}\label{Jacobian}
(1)  $\mathcal{J}_{\nu}$ is given by
\begin{eqnarray*}
\left(\begin{array}{cccccccccc}
1&0&0&0&0&0&0&0&0&0\cr
U_1&\rho&0&0&0&0&0&0&0&0\cr
U_2&0&\rho&0&0&0&0&0&0&0\cr
U_3&0&0&\rho&0&0&0&0&0&0\cr
A^{\nu}_{11}&\frac{1+2\nu}{3}\rho U_1&\frac{1-\nu}{3}\rho U_2&\frac{1-\nu}{3}\rho U_3&\frac{1}{2}\rho&0&0&0&0&0\cr
A^{\nu}_{22}&\frac{1-\nu}{3}\rho U_1&\frac{1+2\nu}{3}\rho U_2&\frac{1-\nu}{3}\rho U_3&0&\frac{1}{2}\rho&0&0&0&0\cr
A^{\nu}_{33}&\frac{1-\nu}{3}\rho U_1&\frac{1-\nu}{3}\rho U_2&\frac{1+2\nu}{3}\rho U_3&0&0&\frac{1}{2}\rho&0&0&0\cr
A^{\nu}_{12}&\frac{\nu}{2}\rho U_2&\frac{\nu}{2}\rho U_1&0&0&0&0&\frac{1}{2}\rho&0&0\cr
A^{\nu}_{23}&0&\frac{\nu}{2}\rho U_3&\frac{\nu}{2}\rho U_2&0&0&0&0&\frac{1}{2}\rho&0\cr
A^{\nu}_{31}&\frac{\nu}{2}\rho U_3&0&\frac{\nu}{2}\rho U_1&0&0&0&0&0&\frac{1}{2}\rho\cr
\end{array}
\right),
\end{eqnarray*}
where $A^{\nu}_{ij}$ is defined as
\begin{eqnarray*}
A^{\nu}_{ii}&=&\frac{1}{2}\left\{\mathcal{T}_{ii}+\frac{(1-\nu)|U|^2+3\nu U^2_i}{3}-1\right\},\cr
A^{\nu}_{ij}&=&\frac{\nu}{2}\left(\mathcal{T}_{ij}+U_iU_j\right).
\end{eqnarray*}
(2) $\mathcal{J}^{-1}_{\nu}$ is given by
\begin{eqnarray*}
\left(\begin{array}{cccccccccc}
1&0&0&0&0&0&0&0&0&0\\
-\frac{U_1}{\rho}&\frac{1}{\rho}&0&0&0&0&0&0&0&0\\
-\frac{U_2}{\rho}&0&\frac{1}{\rho}&0&0&0&0&0&0&0\\
-\frac{U_3}{\rho}&0&0&\frac{1}{\rho}&0&0&0&0&0&0\\
\widetilde{A}_{11}&-\frac{2(1+2\nu)}{3}\frac{U_1}{\rho}&-\frac{2(1-\nu)}{3}\frac{U_2}{\rho}&-\frac{2(1-\nu)}{3}\frac{U_3}{\rho}&\frac{2}{\rho}&0&0&0&0&0\\
\widetilde{A}_{22}&-\frac{2(1-\nu)}{3}\frac{U_1}{\rho}&-\frac{2(1+2\nu)}{3}\frac{U_2}{\rho}&-\frac{2(1-\nu)}{3}\frac{U_3}{\rho}&0&\frac{2}{\rho}&0&0&0&0\\
\widetilde{A}_{33}&-\frac{2(1-\nu)}{3}\frac{U_1}{\rho}&-\frac{2(1-\nu)}{3}\frac{U_2}{\rho}&-\frac{2(1+2\nu)}{3}\frac{U_3}{\rho}&0&0&\frac{2}{\rho}&0&0&0\\
\widetilde{A}_{12}&-\frac{ U_2}{\rho}&-\nu\frac{U_1}{\rho}&0&0&0&0&\frac{2}{\rho}&0&0\\
\widetilde{A}_{23}&0&-\nu\frac{U_3}{\rho}&-\nu\frac{U_2}{\rho}&0&0&0&0&\frac{2}{\rho}&0\\
\widetilde{A}_{31}&-\nu\frac{U_3}{\rho}&0&-\nu\frac{U_1}{\rho}&0&0&0&0&0&\frac{2}{\rho}\\
\end{array}
\right),
\end{eqnarray*}
where $\widetilde{A}^{\nu}_{ij}$ is defined as
\begin{eqnarray*}
\widetilde{A}_{ii}&=&-\frac{1}{\rho}\left\{\mathcal{T}_{ii}+\frac{(1-\nu)|U|^2+3\nu U^2_i}{3}-1\right\},\cr
\widetilde{A}_{ij}&=&\frac{\nu}{\rho}(-\mathcal{T}_{ij}+U_iU_j).
\end{eqnarray*}
(3) When $F=\mu$, $\mathcal{J}_{\nu}$ and $\mathcal{J}^{-1}_{\nu}$ reduce to the following simpler form:
\begin{eqnarray*}
\left.\mathcal{J}_{\nu}~\right|_{F=\mu}=I(4,6;1,1/2)~\mbox{ and }~ \left. \mathcal{J}^{-1}_{\nu}~\right|_{F=\mu}=I(4,6;1,2).
\end{eqnarray*}
For the definition of $I(m,n;a,b)$, see the notation at the end of the introduction.
\end{lemma}
\begin{proof}
The proof is straightforward but very tedious. We omit the proof.
\end{proof}
%
%
%
%
%
\begin{lemma}\label{Diff Macroscopic Field}
We have\newline
$(1)$ Derivatives for $\det\mathcal{T}_{\nu}$: $( 1\leq i,j\leq 3, i\neq j )$
\begin{eqnarray*}
&&\frac{\partial\det\mathcal{T}_{\nu}}{\partial \rho}\Big|_{F=\mu}=0,\quad \frac{\partial\det\mathcal{T}_{\nu}}{\partial U_i}\Big|_{F=\mu}=0,~\cr
&&\frac{\partial\det\mathcal{T}_{\nu}}{\partial \mathcal{T}_{ii}}\Big|_{F=\mu}=1,\quad\frac{\partial\det\mathcal{T}_{\nu}}{\partial \mathcal{T}_{ij}}\Big|_{F=\mu}=0.
\end{eqnarray*}
$(2)$ Derivatives for $\mathcal{M}_{\nu}$: $( 1\leq i,j\leq 3, i\neq j )$
\begin{eqnarray*}
&&\frac{\partial\mathcal{M}_{\nu}}{\partial\rho}\Big|_{F=\mu}=\mu(v),\quad
\frac{\partial\mathcal{M}_{\nu}}{\partial U_i}\Big|_{F=\mu}=v_i\mu(v),\cr
&&\frac{\partial\mathcal{M}_{\nu}}{\partial\mathcal{T}_{ii}}\Big|_{F=\mu}=\left(\frac{v^2_i-1}{2}\right)\mu(v),\quad
\frac{\partial\mathcal{M}_{\nu}}{\partial\mathcal{T}_{ij}}\Big|_{F=\mu}= v_iv_j\mu(v).
\end{eqnarray*}
\end{lemma}
\begin{proof}
(1) A straightforward calculation leads to the following explicit form of the determinant of $\mathcal{T}_{\nu}$:
\begin{eqnarray}\label{detT}
\begin{split}
\det\mathcal{T}_{\nu}&=\mathcal{T}_{11}\mathcal{T}_{22}\mathcal{T}_{33}-\mathcal{T}_{23}^2\mathcal{T}_{11}-\mathcal{T}_{31}^2\mathcal{T}_{22}
-\mathcal{T}^2_{12}\mathcal{T}_{33}.
\end{split}
\end{eqnarray}
Then (1) follows from explicit calculations using
\begin{eqnarray*}
\frac{\partial\mathcal{T}_{ij}}{\partial\rho}=0,~\frac{\partial\mathcal{T}_{ij}}{\partial U}=0,~
\frac{\partial\mathcal{T}_{ij}}{\partial\mathcal{T}_{\ell k}}=\left\{\begin{array}{l}1\quad (i=\ell, j=k)\\0\quad (otherwise)\end{array}\right.,
\end{eqnarray*}
and \begin{eqnarray*}
\mathcal{T}_{ij}\big|_{\mu}=\delta_{ij}.
\end{eqnarray*}
\newline
(2) We only consider $\frac{\partial\mathcal{M}}{\partial\mathcal{T}_{ij}}$. Other terms can be obtained similarly.
We first observe that
\begin{eqnarray*}
\frac{\partial\mathcal{M}_{\nu}}{\partial\mathcal{T}_{ij}}=
\left[-\frac{1}{2}\frac{1}{|\det\mathcal{T}|}\frac{\partial\det\mathcal{T}}{\partial\mathcal{T}_{ij}}
+\frac{1}{2}(v-U)\mathcal{T}^{-1}\left(\frac{\partial\mathcal{T}}{\partial\mathcal{T}_{ij}}\right)\mathcal{T}^{-1}(v-U)\right]
\mathcal{M}_{\nu},
\end{eqnarray*}
When $i=j=1$, we have
\begin{eqnarray*}
\frac{\partial\mathcal{M}_{\nu}}{\partial\mathcal{T}_{11}}\Big|_{F=\mu}&=&
\left\{-\frac{1}{2}+\frac{1}{2}v^{\top}\left(\begin{array}{ccc}1&0&0\cr0&0&0\cr0&0&0\end{array}\right)v\right\}
\mathcal{M}_{\nu}\cr
&=&\left(\frac{v^2_1-1}{2}\right)\mu.
\end{eqnarray*}
$\frac{\partial\mathcal{M}_{\nu}}{\partial\mathcal{T}_{22}}$, $\frac{\partial\mathcal{M}_{\nu}}{\partial\mathcal{T}_{33}}$ can be obtained in the same manner.
In the case $i\neq j$, we observe that
\begin{eqnarray*}
\frac{\partial\mathcal{M}_{\nu}}{\partial\mathcal{T}_{12}}\Big|_{F=\mu}&=&
\left\{0+\frac{1}{2}v^{\top}\left(\begin{array}{ccc}0&1&0\cr1&0&0\cr0&0&0\end{array}\right)v\right\}\mu\cr
&=& v_1v_2\mu.
\end{eqnarray*}
Similarly, we have
\begin{eqnarray*}
\frac{\partial\mathcal{M}}{\partial\mathcal{T}_{23}}\Big|_{F=\mu}=v_2v_3\mu~\mbox{ and }~\frac{\partial\mathcal{M}}{\partial\mathcal{T}_{31}}\Big|_{F=\mu}=v_3v_1\mu.
\end{eqnarray*}
\end{proof}
%
%
%
%
Now, we are ready to prove the main theorem of this section, which basically says that the linearized relaxation operator
is composed of $\nu$-perturbation of the projection on the macroscopic kernel and nonlinear terms.
\begin{theorem}\label{expansion} Let $F=\mu+\sqrt{\mu}f$. Then the ellipsoidal Gaussian $\mathcal{M}_{\nu}(F)$ can be expanded around $\mu$ as follows:
\begin{equation*}
\mathcal{M}_{\nu}(F)=\mu+(P_{\nu}f)\sqrt{\mu}+\sum_{1\leq i,j\leq 3}\left(\int^1_0\big\{D^2_{(\rho_{\theta},\rho_{\theta}U_{\theta},G_{\theta})}\mathcal{M}(\theta)\big\}_{ij}(1-\theta)^2d\theta\right)
\langle f,e_i\rangle_{L^2_v}\langle f,e_j\rangle_{L^2_v},
\end{equation*}
Here, $P_{\nu}$ is given by a $\nu$-perturbation of the usual macroscopic projection $P_0$:
\begin{eqnarray*}
P_{\nu}f\equiv P_0f+\nu ( P_1f+ P_2f),
\end{eqnarray*}
where
\begin{eqnarray*}\begin{split}
P_0f&=\left(\int f\sqrt{\mu} dv\right)\sqrt{\mu} +\left(\int f v\sqrt{\mu}dv\right)\cdot v\sqrt{\mu}
+\left(\int f\frac{|v|^2-3}{\sqrt{6}}\sqrt{\mu}dv\right) \frac{|v|^2-3}{\sqrt{6}}\sqrt{\mu},\cr
P_1f&=\sum^3_{i=1}\left(\int f~ \frac{3v^2_i-|v|^2}{3\sqrt{2}}\sqrt{\mu}dv\right)\frac{3v^2_i-|v|^2}{3\sqrt{2}}\sqrt{\mu},\cr
P_2f&=\sum_{i< j}\left(\int f v_iv_j\sqrt{\mu}dv\right) v_iv_j\sqrt{\mu}.
\end{split}
\end{eqnarray*}
and $\mathcal{M}_{\nu}(\theta)$ denotes
\begin{equation*}
\mathcal{M}_{\nu}(\theta)=\frac{\rho_{\theta}}{\sqrt{\det(2\pi \mathcal{T}_{\theta}})}
\exp\left(-\frac{1}{2}(v-U_{\theta})^{\top}\mathcal{T}^{-1}_{\theta}(v-U_{\theta})\right),
\end{equation*}
where the transitional macroscopic fields $\rho_{\theta}$, $U_{\theta}$, $G_{\theta}$ and $\mathcal{T}_{\theta}$ are defined by
\begin{eqnarray*}
\rho_{\theta}=\theta\rho+(1-\theta), ~ \rho_{\theta}U_{\theta}=\theta \rho U, ~\mbox{and }~  G_{\theta}=\theta G,
\end{eqnarray*}
and
\begin{eqnarray*}
\mathcal{T}_{\theta}=\left(
\begin{array}{ccc}
(1-\nu) T_{\theta}+\nu\Theta_{\theta11}&\nu\Theta_{\theta12}&\nu\Theta_{\theta13}\cr
\nu\Theta_{\theta21}&(1-\nu) T_{\theta}+\nu\Theta_{\theta22}&\nu\Theta_{\theta23}\cr
\nu\Theta_{\theta31}&\nu\Theta_{\theta32}&(1-\nu) T_{\theta}+\nu\Theta_{\theta33}
\end{array}
\right).
\end{eqnarray*}
\end{theorem}
\begin{proof}
We define $g(\theta)$ as
\begin{eqnarray*}
g(\theta)&=&\mathcal{M}\big(~\theta(\rho,\rho U, G)+(1-\theta)\left(1,0^3,0^6\right)\big)\cr
&=&\mathcal{M}\left(\rho_{\theta},\rho_{\theta}U_{\theta},G_{\theta}\right).
\end{eqnarray*}
Note that $g(\theta)$ represents the transition from the global Maxwellian $\mu(v)$ to the ellipsoidal Gaussian $\mathcal{M}_{\nu}(F)$.
Then we have from the Taylor's theorem
\begin{eqnarray}\label{Taylor Expansion}
g(1)=g(0)+g^{\prime}(0)+\int^1_0g^{\prime\prime}(\theta)(1-\theta)^2d\theta.
\end{eqnarray}
The first term in the right hand side is the global Maxwellian: $g(0)=\mu$.
We now consider the second and the third terms:\newline\newline
(i) $g^{\prime}(0)$:  We observe from Lemma \ref{Diff Macroscopic Field} that
\begin{eqnarray*}\begin{split}
D_{(\rho, U, \mathcal{T}_{\nu})}\mathcal{M}_{\nu}(0)
&=\Big(\frac{\partial\mathcal{M}_{\nu}}{\partial \rho}, \frac{\partial\mathcal{M}_{\nu}}{\partial U},
\frac{\partial\mathcal{M}_{\nu}}{\partial \mathcal{T}_{\nu}}\Big)\Big|_{F=\mu}\cr
&=\Big(1,v,\frac{v^2_1-1}{2}, \frac{v^2_2-1}{2}, \frac{v^2_3-1}{2},  v_1v_2,  v_2v_3,  v_3v_1\Big)\mu(v).
\end{split}
\end{eqnarray*}
Then, using the identities in Lemma \ref{Jacobian} and Lemma \ref{Diff Macroscopic Field}, $g^{\prime}(0)$ can be represented as
\begin{eqnarray*}
g^{\prime}(0)&=&\frac{d}{d\theta}\mathcal{M}\big(\theta(\rho,\rho U, G)+(1-\theta)(1,0,0,0,0^6)\big)\Big|_{\theta=0}\cr
&=&(\rho-1,\rho U, G)^{\top}\cdot \mathcal{J}^{-1}_{\theta}
D_{(\rho_{\theta},U_{\theta},\mathcal{T}_{\theta})}\mathcal{M}\Big|_{\theta=0}\cr
&=&(\rho-1,\rho U, G)^{\top}\cdot \mathcal{J}^{-1}\cr
&&\times\left(1, v,\frac{v^2_1-1}{2},\frac{v^2_3-1}{2},\frac{v^2_3-1}{2},  v_1v_2,  v_2v_3,  v_3v_1\right)\mu\cr
&=&\left(\int f\sqrt{\mu}dv\right)\mu+\left(\int fv\sqrt{\mu}dv\right)\cdot v\mu\cr
&+&2\sum_{i=1}^3G_{ii}\left(\frac{v^2_i-1}{2}\right)\mu
+2\sum_{i<j}G_{ij}v_iv_j\mu.
\end{eqnarray*}
Here $\mathcal{J}_{\theta}$ denotes $\frac{\partial(\rho_{\theta} U_{\theta}, G_{\theta})}{\partial(\rho_{\theta} U_{\theta},\mathcal{T}_{\theta})}$
and we used $\mathcal{J}_{0}=\mathcal{J}$.\newline
(ii) $g^{\prime\prime}(\theta)$: By an explicit computation, we find
\begin{eqnarray*}
~g^{\prime\prime}(\theta)&=&\frac{d^2\mathcal{M}}{d\theta^2}(\theta(\rho-1,\rho U, G)+(1-\theta)(1,0,0^6))\cr
&=&(\rho-1,\rho U, G)^{\top}\left\{D^2_{\big(\rho_{\theta},\rho_{\theta} U_{\theta}, G_{\theta})}\mathcal{M}(\theta)\right\}(\rho-1,\rho U, G).
\end{eqnarray*}
(iii) We claim that
\[
g^{\prime}(0)=P_{\nu}f\sqrt{\mu}.
\]
Note that it is enough to establish
\begin{eqnarray*}
2\sum_{i=1}^3G_{ii}\left(\frac{v^2_i-1}{2}\right)\sqrt{\mu}&=&\left(\int f\frac{|v|^2-3}{\sqrt{6}}\sqrt{\mu}dv\right) \frac{|v|^2-3}{\sqrt{6}}\sqrt{\mu}\cr
&+&\sum^3_{i=1}\left(\int f~ \frac{3v^2_i-|v|^2}{3\sqrt{2}}\sqrt{\mu}dv\right)\frac{3v^2_i-|v|^2}{3\sqrt{2}}\sqrt{\mu}.
\end{eqnarray*}
We first observe that $G_{ii}$ $(i=1,2,3)$ can be decomposed as
\begin{eqnarray*}
G_{ii}&=&\frac{1-\nu}{3}\int_{\mathbb{R}^3} f \frac{|v|^2}{2}dv+\nu\int_{\mathbb{R}^3} f \frac{v^2_i}{2}dv-\int_{\mathbb{R}^3} \frac{1}{2}fdv\cr
&=&\int_{\mathbb{R}^3} f\left\{\left(\frac{|v|^2-3}{6}\right)+\nu\left(\frac{3v_i^2-|v|^2}{6}\right)\right\}\sqrt{\mu}dv,
\end{eqnarray*}
so that
\begin{eqnarray*}
\begin{split}
\hspace{0.2cm} 2\sum_{i=1}^3\int_{\mathbb{R}^3} fG_{ii} \frac{v^2_i-1}{2}\sqrt{\mu}
&=2\sum_{i=1}^3\left(\int_{\mathbb{R}^3}f\left[\frac{|v|^2-3}{6}+\nu \left(\frac{3v_i^2-|v|^2}{6}\right)\right]\sqrt{\mu}dv\right)
\frac{v^2_i-1}{2}\sqrt{\mu}\cr
&=2\sum_{i=1}^3\left(\int_{\mathbb{R}^3}f\frac{|v|^2-3}{6}\sqrt{\mu}dv\right)\frac{v^2_i-1}{2}\sqrt{\mu}\cr
&+2\nu\sum_{i=1}^3\left(\int_{\mathbb{R}^3}f\frac{3v_i^2-|v|^2}{6}\sqrt{\mu}dv\right)
\frac{v^2_i-1}{2}\sqrt{\mu}\cr
&+A+B.
\end{split}
\end{eqnarray*}
We compute $A$ as
\begin{eqnarray*}
A&=&2\left(\int_{\mathbb{R}^3}f\frac{|v|^2-3}{6}\sqrt{\mu}dv\right)\sum^3_{i=1}\left(\frac{v^2_i-1}{2}\right)\sqrt{\mu}\cr
&=&2\left(\int_{\mathbb{R}^3}f\frac{|v|^2-3}{6}\sqrt{\mu}dv\right)\left(\frac{|v|^2-3}{2}\right)\sqrt{\mu}\cr
&=&\left(\int_{\mathbb{R}^3}f\frac{|v|^2-3}{\sqrt{6}}\sqrt{\mu}dv\right)\left(\frac{|v|^2-3}{\sqrt{6}}\right)\sqrt{\mu}.
\end{eqnarray*}
For $B$, we observe that
\[
\frac{v^2_i-1}{2}=\frac{3v^2_i-|v|^2}{6}+\frac{|v|^2-3}{6},
\]
and
\[
\sum^3_{i=1}\frac{3v^2_i-|v|^2}{6}=0,
\]
to derive
\begin{eqnarray*}
B&=&2\sum^3_{i=1}\left(\int_{\mathbb{R}^3}f\frac{3v^2_i-|v|^2}{6}\sqrt{\mu}dv\right)\left(\frac{3v^2_i-|v|^2}{6}+\frac{|v|^2-3}{6}\right)\sqrt{\mu}\cr
&=&2\sum^3_{i=1}\left(\int_{\mathbb{R}^3}f\frac{3v^2_i-|v|^2}{6}\sqrt{\mu}dv\right)\left(\frac{3v^2_i-|v|^2}{3}\right)\sqrt{\mu}\cr
&+&2\sum^3_{i=1}\left(\int_{\mathbb{R}^3}f\frac{3v^2_i-|v|^2}{6}\sqrt{\mu}dv\right)\left(\frac{|v|^2-3}{6}\right)\sqrt{\mu}\cr
&=&\sum^3_{i=1}\left(\int_{\mathbb{R}^3}f\frac{3v^2_i-|v|^2}{3\sqrt{2}}\sqrt{\mu}dv\right)\left(\frac{3v^2_i-|v|^2}{3\sqrt{2}}\right)\sqrt{\mu}\cr
&+&2\Big(\int_{\mathbb{R}^3}f\underbrace{\sum^3_{i=1}\frac{3v^2_i-|v|^2}{6}}_{=0}\sqrt{\mu}dv\Big)\left(\frac{|v|^2-3}{6}\right)\sqrt{\mu}\cr
&=&\sum^3_{i=1}\left(\int_{\mathbb{R}^3}f\frac{3v^2_i-|v|^2}{3\sqrt{2}}\sqrt{\mu}dv\right)\left(\frac{3v^2_i-|v|^2}{3\sqrt{2}}\right)\sqrt{\mu}.
\end{eqnarray*}
Plugging (i), (ii), (iii) into \ref{Taylor Expansion}, we obtained the desired result.
\end{proof}
We now consider the linearization of the collision frequency.
%
%
%
%

\begin{lemma}The collision frequency $A_{\nu}$ can be linearized around the normalized global Maxwellian as follows:
\begin{equation*}
A_{\nu}=\frac{1}{1-\nu}+\frac{1}{1-\nu}A_p,
\end{equation*}
where
\begin{equation*}
A_p=\left\{\int^1_0\mathcal{J}^{-1}_{\theta}(T_{\theta},0^3, 1/3\rho_{\theta}Id)(1-\theta)d\theta\right\}\cdot(\rho-1,\rho U, G).
\end{equation*}
\end{lemma}
\begin{proof}
We expand $A_{\nu}$ by the Taylor's theorem. Then the second term reads
\begin{equation*}
\left\{\int^1_0D_{(\rho_{\theta},\rho_{\theta}U_{\theta},G_{\theta})}(\rho_{\theta}T_{\theta})(1-\theta)d\theta\right\}\cdot(\rho-1,\rho U, G).
\end{equation*}
Note that
\begin{equation*}
D_{(\rho_{\theta},\rho_{\theta}U_{\theta},G_{\theta})}=\mathcal{J}^{-1}_{\theta}D_{(\rho_{\theta},U_{\theta},\mathcal{T}_{\theta})},
\end{equation*}
to see
\begin{eqnarray*}
A_p&=&\left\{\int^1_0\mathcal{J}^{-1}_{\theta}D_{(\rho_{\theta},U_{\theta},\mathcal{T}_{\theta})}(\rho_{\theta}T_{\theta})(1-\theta)d\theta\right\}
\cdot(\rho-1,\rho U, G)\cr
&=&\left\{\int^1_0\mathcal{J}^{-1}_{\theta}(T_{\theta},0^3, 1/3\rho_{\theta}Id)(1-\theta)d\theta\right\}\cdot(\rho-1,\rho U, G).
\end{eqnarray*}
\end{proof}
Instead of writing down $D^2_{(\rho_{\theta}, \rho_{\theta} U_{\theta}, G_{\theta})}$ explicitly,
we introduce generic notations which considerably simplify the argument. We first observe that
\begin{eqnarray*}
D^2_{(\rho_{\theta}, \rho_{\theta} U_{\theta}, G_{\theta})}\mathcal{M}(\theta)=
\mathcal{J}^{-1}_{\theta}D_{(\rho_{\theta}, U_{\theta}, \mathcal{T}_{\theta})}\mathcal{J}^{-1}_{\theta}
D_{(\rho_{\theta}, U_{\theta}, \mathcal{T}_{\theta})}
\mathcal{M}(\theta).
\end{eqnarray*}
We then invoke Lemma \ref{Jacobian} to conclude the following lemma.
\begin{lemma}
There exist generic polynomials $P^{\mathcal{M}}_{i,j}$, $R^{\mathcal{M}}_{i,j}$ such that
\begin{eqnarray*}
&&(\rho-1,\rho U, G)^{\top}\left\{D^2_{(\rho_{\theta},\rho_{\theta}U_{\theta},G_{\theta})}\mathcal{M}(\theta)\right\}(\rho-1,\rho U, G)\cr
&&\qquad=\sum_{i,j}\frac{P^{\mathcal{M}}_{i,j}(\rho_{\theta}, v-U_{\theta},U_{\theta}, \mathcal{T}^{-1}_{\theta},\nu)}{R^{\mathcal{M}}_{i,j}(\rho_{\theta},\det\mathcal{T}_{\theta})}
\exp\left(-\frac{1}{2}(v-U_{\theta})^{\top}\mathcal{T}^{-1}_{\theta}(v-U_{\theta})\right)\langle f,e_i\rangle\langle f,e_j\rangle,
\end{eqnarray*}
where $P^{\mathcal{M}}_{i,j}(x_1,\ldots,x_n)$ and $R^{\mathcal{M}}_{i,j}(x_1,\ldots, x_n)$ satisfy the following structural assumptions $(\mathcal{H}_{\mathcal{M}})$:
\begin{itemize}
\item $(\mathcal{H}_{\mathcal{M}_{\nu}}1)$ $P^{\mathcal{M}_{\nu}}_{i,j}$ is a polynomial such that $P_{i,j}(0,0,\ldots, 0)=0.$
\item $(\mathcal{H}_{\mathcal{M}_{\nu}}2)$ $R^{\mathcal{M}_{\nu}}_{i,j}$ is a monomial.
\end{itemize}
\end{lemma}
\begin{lemma}
There exist generic polynomials $P^{A_{\nu}}_{i}$, $R^{A_{\nu}}_{i}$ such that
\begin{eqnarray*}
\left\{\mathcal{J}_{\theta}^{-1}(T_{\theta},0^3,1/3\rho_{\theta}Id)\right\}\cdot(\rho-1,\rho U, G)
=\sum_{i}\frac{P^{A_{\nu}}_{i}(\rho_{\theta},U_{\theta}, \mathcal{T}_{\nu\theta},\nu)}{R^{A^{\nu}}_{i}(\rho_{\theta})}
\langle f,e_i\rangle,
\end{eqnarray*}
where $P^{A^{\nu}}_{i,j}(x_1,\ldots,x_n)$ and $R^{A^{\nu}}_{i,j}(x_1,\ldots, x_n)$ satisfy the following structural assumptions $(\mathcal{H}_{A_{\nu}})$.
\begin{itemize}
\item $(\mathcal{H}_{A_{\nu}}1)$ $P^{A_{\nu}}_{i}$ is a polynomial such that $P_{i,j}(0,0,\ldots, 0)=0.$
\item $(\mathcal{H}_{A_{\nu}}2)$ $
R^{A_{\nu}}_{i}$ is a monomial.
\end{itemize}
\end{lemma}
Note that $P^{\mathcal{M}}_{ij}$, $R^{\mathcal{M}}_{ij}$, $P^{\mathcal{M}}_{i}$ and $R^{A_{\nu}}_{i}$ are defined generically.
They may change line after line during the argument. But explicit form is not important as long as we keep in mind the structural assumptions $\mathcal{H}_{\mathcal{M}}$
and $\mathcal{H}_{A_{\nu}}$.
To simplify the notation further, we define $\mathcal{Q}^{\mathcal{M}}_{ij}$ and $\mathcal{Q}^{A_{\nu}}_{i}$ as
\begin{equation*}
\mathcal{Q}^{\mathcal{M}_{\nu}}_{ij}(\theta)=\frac{1}{\sqrt{\mu}}\int^1_0\frac{P^{\mathcal{M}_{\nu}}_{i,j}(\rho_{\theta}, v-U_{\theta}, \mathcal{T}^{-1}_{\theta},\nu)}{R^{\mathcal{M}}_{ij}(\rho_{\theta},\det\mathcal{T}_{\nu\theta})}\exp\left(-\frac{1}{2}(v-U_{\theta})^{\top}\mathcal{T}^{-1}_{\nu\theta}(v-U_{\theta})\right)
(1-\theta)^2d\theta
\end{equation*}
and
\begin{equation*}
\mathcal{Q}^{A_{\nu}}_{i}(\theta)=\int^1_0\frac{P^{A_{\nu}}_{i}(\rho_{\theta}, U_{\theta}, \mathcal{T}_{\theta},\nu)}{R^{A_{\nu}}_{i}(\rho_{\theta})}
(1-\theta)d\theta.
\end{equation*}
Then the relaxation operator and the collision frequency can be expressed in a more succinct form:
\begin{equation*}
\mathcal{M}_{\nu}(F)-F=\big(P_{\nu}f-f\big)\sqrt{\mu}+\sum \mathcal{Q}^{\mathcal{M}_{\nu}}_{ij}\langle f,e_i\rangle_{L^2_v}\langle f,e_i\rangle_{L^2_v},
\end{equation*}
and
\begin{equation*}
A_{\nu}=\frac{1}{1-\nu}+\frac{1}{1-\nu}\sum \mathcal{Q}^{\nu}_{i}\langle f,e_i\rangle_{L^2_v}.
\end{equation*}

We summarize the result in the following proposition.
\begin{proposition}\label{Mf-f} The relaxation operator can be linearized around the normalized global Maxwellian $\mu$ as follows
\begin{equation*}
A_{\nu}\big(\mathcal{M}_{\nu}(F)-F\big)=\frac{1}{1-\nu}\Big(1+\sum_{i}\mathcal{Q}^{A_{\nu}}_{i}\langle f,e_i\rangle\Big)
\Big\{(P_{\nu}f-f)+\sum_{i,j}\mathcal{Q}^{\mathcal{M}_{\nu}}_{ij}\langle f,e_i\rangle\langle f,e_j\rangle\Big\}\sqrt{\mu}.
\end{equation*}
\end{proposition}
We now substitute the standard perturbation $F=\mu+\sqrt{\mu}f$ into (1.1) and apply proposition \ref{Mf-f} to obtain the perturbed ES-BGK model:
\begin{eqnarray}\label{LBGK}
\begin{split}
\partial_t f+v\cdot\nabla_x f&=L_{\nu}f+\Gamma(f),\cr
f(x,v,0)&=f_0(x,v),
\end{split}
\end{eqnarray}
where $f_0(x,v)=\frac{F_0-\mu}{\sqrt{\mu}}$. The linearized relaxation operator $L_{\nu}$ and the nonlinear perturbation $\Gamma(f)$ are defined
as follows:
\begin{eqnarray}
L_{\nu}f=\frac{1}{1-\nu}\big(P_{\nu}f-f\big),
\end{eqnarray}
and
\begin{eqnarray*}
\Gamma(f)&=&\frac{1}{1-\nu}\Big\{\sum_{i}\mathcal{Q}^{A}_{i}\langle f,e_i\rangle\Big\}\big(P_{\nu}f-f\big)\cr
&+&\frac{1}{1-\nu}\sum_{1\leq i,j \leq 3}\mathcal{Q}^{\mathcal{M}}_{i,j}\langle f,e_i\rangle_{L^2_v}\langle f,e_j\rangle_{L^2_v}\cr
&+&\frac{1}{1-\nu}\sum_{1\leq i,j \leq 3}\mathcal{Q}^{A^{\nu}}_i\mathcal{Q}^{\mathcal{M}}_{j,k}\langle f,e_i\rangle_{L^2_v}
\langle f,e_j\rangle_{L^2_v}
\langle f,e_k\rangle_{L^2_v}\cr
&\equiv&\Gamma_1(f,f)+\Gamma_2(f,f)+\Gamma_3(f,f,f).
\end{eqnarray*}
The conservation laws in (\ref{ConservationLawsF}) now take the following form:
\begin{eqnarray*}\label{ConservationLawsf0}
\begin{split}
\int_{\mathbb{T}^3_x\times\mathbb{R}^3_v}f(x,v,t)\sqrt{\mu} ~dxdv&=\int_{\mathbb{T}^3_x\times\mathbb{R}^3_v}f_0(x,v)\sqrt{\mu} ~dxdv,\cr
\int_{\mathbb{T}^3_x\times\mathbb{R}^3_v}f(x,v,t)v\sqrt{\mu} ~dxdv&=\int_{\mathbb{T}^3_x\times\mathbb{R}^3_v}f_0(x,v)v\sqrt{\mu} ~dxdv,\cr
\int_{\mathbb{T}^3_x\times\mathbb{R}^3_v}f(x,v,t)|v|^2\sqrt{\mu} ~dxdv&=\int_{\mathbb{T}^3_x\times\mathbb{R}^3_v}f_0(x,v)|v|^2\sqrt{\mu} ~dxdv.
\end{split}
\end{eqnarray*}
Therefore, if initial data shares the same mass, momentum and energy with $\mu$, the conservation laws read
\begin{eqnarray}\label{ConservationLawsf}
\begin{split}
\int_{\mathbb{T}^3_x\times\mathbb{R}^3_v}f(x,v,t)\sqrt{\mu} ~dxdv&=0,\cr
\int_{\mathbb{T}^3_x\times\mathbb{R}^3_v}f(x,v,t)v\sqrt{\mu} ~dxdv&=0,\cr
\int_{\mathbb{T}^3_x\times\mathbb{R}^3_v}f(x,v,t)|v|^2\sqrt{\mu} ~dxdv&=0.
\end{split}
\end{eqnarray}

\subsection{Analysis of the linearized relaxation operator}
We now study the dissipative mechanism of the linearized operator. We start with the following technical lemma.
%
%
%
%
\begin{lemma}\label{Nilpotent} $P_0$, $P_1$ and $P_2$ satisfies the following properties:
\begin{enumerate}
\item $P_0$, $P_2$ and $P_2$ are orthonormal projections:
\begin{eqnarray*}
~P^2_0=P_0,~P^2_1=P_1,~P^2_2=P_2.
\end{eqnarray*}
\item $P_0$, $P_1$ and $P_3$ are mutually orthogonal in the following sense:
\begin{eqnarray*}
P_0P_1=P_1P_0=P_0P_2=P_2P_0=P_1P_2=P_2P_1=0.
\end{eqnarray*}
\end{enumerate}
\end{lemma}
\begin{proof}
(1) The first and third identities $P^2_0=P_0$ and $P^2_2=P_2$ follow from the fact that $\{\sqrt{\mu},v\sqrt{\mu},|v|^2\sqrt{\mu}\}$ and $\{v_1v_2\sqrt{\mu}, v_2v_3\sqrt{\mu}, v_3v_1\sqrt{\mu}\}$ form orthonormal bases respectively. To show $P^2_1=P_1$, we first observe that
\begin{eqnarray*}
&&\left\langle (3v^2_i-|v|^2)\sqrt{\mu}, (3v^2_i-|v|^2)\sqrt{\mu}\right\rangle_{L^2_v}=12,\hspace{0.6cm}(1=1,2,3)\cr
&&\left\langle (3v^2_i-|v|^2)\sqrt{\mu}, (3v^2_j-|v|^2)\sqrt{\mu}\right\rangle_{L^2_v}=-6 \quad\quad(i\neq j).
\end{eqnarray*}
Using this, we have for $c_i(v)=(3v^2_i-|v|^2)/3\sqrt{2}$
\begin{eqnarray*}
P_1^2f&=&P_1\big\{\langle f,c_1\rangle_{L^2_v} c_1+\langle f,c_2\rangle_{L^2_v} c_2+\langle f,c_3\rangle_{L^2_v} c_3\big\}\cr
&=&\frac{1}{3}\{2\langle f,c_1\rangle_{L^2_v}-\langle f,c_2\rangle_{L^2_v}-\langle f,c_3\rangle_{L^2_v}\}c_1\cr
&+&\frac{1}{3}\{-\langle f,c_1\rangle_{L^2_v}+2\langle f,c_2\rangle_{L^2_v}-\langle f,c_3\rangle_{L^2_v}\}c_1\cr
&+&\frac{1}{3}\{-\langle f,c_1\rangle_{L^2_v}-\langle f,c_2\rangle_{L^2_v}+2\langle f,c_3\rangle_{L^2_v}\}c_1\cr
&=&\Big\langle f,\frac{2c_1-c_2-c_3}{3}\Big\rangle_{L^2_v} c_1+\Big\langle f,\frac{-c_1+2c_2-c_3}{3}\Big\rangle_{L^2_v} c_2+\Big\langle f,\frac{-c_1-c_2+2c_3}{3}
\Big\rangle_{L^2_v} c_3\cr
&=&\langle f,c_1\rangle_{L^2_v} c_1+\langle f,c_2\rangle_{L^2_v} c_2+\langle f,c_3\rangle_{L^2_v} c_3\cr
&=&P_1f.
\end{eqnarray*}
In the last line, we used $c_1+c_2+c_3=0$.\newline
(2) Straightforward calculations gives
\begin{eqnarray*}
\langle \sqrt{\mu},(3v^2_i-|v|^2)\sqrt{\mu}\rangle_{L^2_v}=\langle v_{\ell}\sqrt{\mu},(3v^2_i-|v|^2)\sqrt{\mu}\rangle_{L^2_v}
=\langle (|v|^2-3)\sqrt{\mu}, (3v^2_i-|v|^2)\sqrt{\mu}\rangle_{L^2_v}=0,
\end{eqnarray*}
and
\begin{eqnarray*}
\langle v_iv_j\sqrt{\mu}, (3v^2_k-|v|^2)\sqrt{\mu}\rangle_{L^2_v}=0.
\end{eqnarray*}
This implies (2).
\end{proof}
We now prove the main theorem of this section. Note that that the estimate is uniform with respect to $\nu$.
%
%
%
%
\begin{theorem}\label{degenerate coercivity} For $-\frac{1}{2}<\nu<1$, we have
\begin{eqnarray*}
\langle L_{\nu}f,f\rangle_{L^2_v}\leq -\min\Big\{1,\frac{1-|\nu|}{1-\nu}\Big\}\|(I-P_0)f\|^2_{L^2_{v}}.
\end{eqnarray*}
\end{theorem}
\begin{proof}
From the definition of $L_{\nu}$, we have
\begin{eqnarray}\label{setting}
\begin{split}
(1-\nu)\langle L_{\nu}f,f\rangle_{L^2_{v}}&=\langle P_{\nu}f-f,f\rangle_{L^2_{v}}\cr
&=\langle P_0f-f+\nu(P_1+P_2)f,f\rangle_{L^2_{v}}\cr
&=-\|(I-P_0)f\|^2_{L^2_v}+\nu\langle(P_1+P_2)f,f\rangle_{L^2_{v}}.
\end{split}
\end{eqnarray}
We recall from Lemma \ref{Nilpotent} that $(P_1+ P_2)\perp P_0$, which gives
\begin{eqnarray}\label{setting2}
\begin{split}
\langle(P_1+P_2)f,f\rangle_{L^2_{v}}&=\langle (P_1+P_2)(I-P_0)f,(I-P_0)f\rangle_{L^2_{v}}\cr
&=\|(P_1+P_2)(I-P_0)f\|^2_{L^2_v}.
\end{split}
\end{eqnarray}
We then observe from Lemma \ref{Nilpotent} that $P_1+P_2$ is a projection operator:
\begin{eqnarray*}
(P_1+P_2)^2=P^2_1+P_1P_2+P_2P_1+P^2_2=P_1+P_2,
\end{eqnarray*}
which leads to
\begin{eqnarray}\label{setting3}
\|(P_1+P_2)(I-P_0)f\|_{L^2_v}\leq\|(I-P_0)f\|^2_{L^2_v}.
\end{eqnarray}
Therefore, we have from (\ref{setting}) - (\ref{setting2})
\begin{eqnarray*}
(1-\nu)\langle L_{\nu}f,f\rangle_{L^2_{v}}\leq -\min\{(1-\nu),(1-|\nu|)\}\|(I-P_0)f\|^2_{L^2_{v}}.
\end{eqnarray*}
Since $(1-\nu)>0$, this completes the proof.
\end{proof}
%
%
%
%
\begin{corollary}\label{Kernel L} For $-1/2<\nu<1$, the kernel of the linearized relaxation operator is given by
\begin{eqnarray*}
Ker\{L_{\nu}\}=Ker\{L_{0}\}=span\{\sqrt{\mu},v\sqrt{\mu},|v|^2\sqrt{\mu}\}.
\end{eqnarray*}
\end{corollary}
\section{Estimates on the macroscopic field}
\subsection{Estimates on the macroscopic field}
To control the nonlinear perturbation $\Gamma(f)$ in the energy norm, we first need to establish
various estimates for macroscopic quantities. Throughout this section, $C_{\nu}>0$ means that $C_{\nu}$ is strictly
positive for all $-1/2<\nu<1$.
\begin{lemma}\label{ULestimatesofRhoUT}
Let $\mathcal{E}(t)$ be sufficiently small, then there exists a positive constant $C>0$ and $C_{\nu}>0$ such that
\begin{eqnarray*}
&&(1)~|\rho(x,t)-1|\leq C\sqrt{\mathcal{E}(t)},\cr
&&(2)~|U_i(x,t)|\leq C\sqrt{\mathcal{E}(t)}, \hspace{1.7cm}(1\leq i\leq 3)\cr
&&(3)~|\mathcal{T}_{\nu ii}(x,t)-1|\leq C_{\nu}\sqrt{\mathcal{E}(t)},\quad \quad(1\leq i\leq 3)\cr
&&(4)~|\mathcal{T}_{ij}(x,t)|\leq \nu C\sqrt{\mathcal{E}(t)},\hspace{1.4cm}(1\leq i<j\leq 3)
\end{eqnarray*}
\end{lemma}
\begin{proof}
(1) We have from H\"{o}lder inequality
\begin{eqnarray*}
|\rho(x,t)-1|=\int_{\mathbb{R}^3} f\sqrt{\mu}dv\leq \|f\|_{L^2_x}\leq \sqrt{\mathcal{E}(t)}.
\end{eqnarray*}
(2) Using the lower bound estimate of $\rho$, H\"{o}lder inequality and $\int_{\mathbb{R}^3} \mu vdv=0$, we see that
\begin{eqnarray*}
|U_i|&=&\frac{1}{\rho}\left|\int_{\mathbb{R}^3} fv_i\sqrt{\mu}dv\right|
\leq\frac{\|f\|_{L^2_{x,v}}}{1-\sqrt{\mathcal{E}(t)}}\cr
&\leq&\frac{\sqrt{\mathcal{E}(t)}}{1-\sqrt{\mathcal{E}(t)}}\leq C\sqrt{\mathcal{E}(t)}.
\end{eqnarray*}
(3) For the upper bound of $\mathcal{T}_{ii}$, we compute as follows:
\begin{eqnarray*}
\mathcal{T}_{ii}&=&(1-\nu)T+\nu\Theta_{ii}\cr
&=&\frac{(1-\nu)}{3}\left\{\frac{1}{\rho}\int_{\mathbb{R}^3}(\mu+\sqrt{\mu}f)|v|^2dv-|U|^2\right\}
+\nu\left\{\frac{1}{\rho}\int_{\mathbb{R}^3}(\mu+\sqrt{\mu}f)v^2_idv-U^2_i\right\}\cr
&\leq&\frac{(1-\nu)}{3\rho}\left\{3+\int_{\mathbb{R}^3}f|v|^2\sqrt{\mu}dv\right\}
+\frac{\nu}{\rho}\left\{1+\int_{\mathbb{R}^3}fv_i^2\sqrt{\mu}dv\right\}\cr
&\leq&\frac{1-\nu}{3}\left\{\frac{3+\sqrt{15}\|f\|_{L^2_{x,v}}}{\rho}\right\}+\nu\left\{\frac{1+\sqrt{3}\|f\|_{L^2_{x,v}}}{\rho}\right\}\cr
&\leq&\frac{1+C_{\nu}\|f\|_{L^2_{x,v}}}{\rho}\cr
&\leq&\frac{1+C_{\nu}\sqrt{\mathcal{E}(t)}}{1-\sqrt{\mathcal{E}}(t)}.
\end{eqnarray*}
Therefore,
\begin{eqnarray}\label{ThetaUpper}
\mathcal{T}_{ii}-1\leq\frac{C\sqrt{\mathcal{E}(t)}}{1-\sqrt{\mathcal{E}(t)}}\leq C\sqrt{\mathcal{E}(t)}.
\end{eqnarray}
Using the lower bound estimate for $\rho$ and $U_i$, we estimate the lower bound similarly as
\begin{eqnarray*}
\mathcal{T}_{ii}&=&(1-\nu)T+\nu\theta_{ii}\cr
&=&\frac{(1-\nu)}{3\rho}\left\{\int_{\mathbb{R}^3}(\mu+\sqrt{\mu}f)|v|^2dv-|U|^2\right\}
+\frac{\nu}{\rho}\left\{\int_{\mathbb{R}^3}(\mu+\sqrt{\mu}f)v^2_idv-U^2_i\right\}\cr
&=&\frac{(1-\nu)}{3\rho}\left\{3+\int_{\mathbb{R}^3}f|v|^2\sqrt{\mu}dv-|U|^2\right\}
+\frac{\nu}{\rho}\left\{1+\int_{\mathbb{R}^3}fv_i^2\sqrt{\mu}dv-U^2_i\right\}\cr
&\geq&\frac{1-\nu}{3\rho}\left\{3-\sqrt{15}\|f\|_{L^2_{x,v}}-C\mathcal{E}(t)\right\}+\frac{\nu}{\rho}\left\{1-\sqrt{3}\|f\|_{L^2_{x,v}}-C\mathcal{E}(t)\right\}\cr
&\geq&\frac{1-C_{\nu}\|f\|_{L^2_{x,v}}-C\mathcal{E}(t)}{\rho}\cr
&\geq&\frac{1-C_{\nu}\sqrt{\mathcal{E}(t)}-CE(t)}{1+\sqrt{\mathcal{E}}(t)}\cr
&\geq&\frac{1-C_{\nu}\sqrt{\mathcal{E}(t)}}{1+\sqrt{\mathcal{E}}(t)}.
\end{eqnarray*}
Hence we have
\begin{eqnarray}\label{ThetaLower}
\mathcal{T}_{ii}-1\geq\frac{-C_{\nu}\sqrt{\mathcal{E}}(t)}{1+\sqrt{\mathcal{E}(t)}}\geq -C_{\nu}\sqrt{\mathcal{E}(t)}.
\end{eqnarray}
(\ref{ThetaUpper}) and (\ref{ThetaLower}) give the desired result for $\mathcal{T}_{ii}$ $(i=1,2,3)$.\newline
(4) $\mathcal{T}_{ij}$ can be estimated similarly as
\begin{eqnarray*}
|\mathcal{T}_{ij}|
&\leq&\frac{\nu}{\rho}\Big|\int_{\mathbb{R}^3}fv_iv_j\sqrt{\mu}dv\Big|+\nu|U_i||U_j|\cr
&\leq&\frac{\nu\|f\|_{L^2_{x,v}}}{1-\sqrt{\mathcal{E}(t)}}+\nu C\mathcal{E}(t)\cr
&\leq&\frac{\nu\sqrt{\mathcal{E}(t)}}{1-\sqrt{\mathcal{E}(t)}}+\nu C\mathcal{E}(t)\cr
&\leq&\nu C\sqrt{\mathcal{E}(t)}.
\end{eqnarray*}
\end{proof}
%
%
%
%
%
\begin{lemma}\label{ULestimatesofRhoUT1}
Suppose $\mathcal{E}(t)$ is sufficiently small. Then there exists a positive constant $C_{|\alpha|}>0$ and $C_{|\alpha|,\nu}>0$  such that
\begin{eqnarray*}
&&(1)~|\partial^{\alpha}\rho(x,t)|\leq \sqrt{\mathcal{E}(t)},\cr
&&(2)~|\partial^{\alpha}U(x,t)|\leq C_{|\alpha|}\mathcal{E}(t),\cr
&&(3)~|\partial^{\alpha}\mathcal{T}_{ij}(x,t)|\leq C_{|\alpha|,\nu}\mathcal{E}(t).
\end{eqnarray*}
\end{lemma}
\begin{proof}
(1) Since $\partial^{\alpha}\mu=0$, we have
\begin{eqnarray*}
|\partial^{\alpha}\rho|&=&\left|\partial^{\alpha}\left(\int_{\mathbb{R}^3} \mu+f\sqrt{\mu}dv\right)\right|=\int |\partial^{\alpha} f|\sqrt{\mu}dv\cr
&\leq&\|\partial^{\alpha}f\|_{L^2_{x,v}}\leq \sqrt{\mathcal{E}(t)}.
\end{eqnarray*}
(2) A straightforward computation using $U=\frac{1}{\rho}\int f v\sqrt{\mu}dv$ and the chain rule gives to
\begin{eqnarray*}
\displaystyle|\partial^{\alpha}U|\leq
\frac{C_{|\alpha|}}{\rho^{2|\alpha|}}
\left(\sum_{|\alpha_1|\leq N}\int_{\mathbb{R}^3}|\partial^{\alpha_1}f||v|\sqrt{\mu}dv\right)\left(1+\sum_{|\alpha_2|\leq N}|\partial^{\alpha_2}\rho|\right)^{|\alpha|}.
\end{eqnarray*}
Then the use of H\"{o}lder inequality and the estimate (1) leads to
\begin{eqnarray*}
|\partial^{\alpha}U|\leq\displaystyle C_{|\alpha|}\frac{\sqrt{\mathcal{E}(t)}\big(1+\sqrt{\mathcal{E}(t)}\big)^{|\alpha|}}{\big(1-\mathcal{E}(t)\big)^{2|\alpha|}}
\leq C_{|\alpha|}\sqrt{\mathcal{E}(t)}.
\end{eqnarray*}
(3) Recall $\mathcal{T}_{ij}=\frac{1-\nu}{3\rho}\left\{\int f|v|^2\sqrt{\mu}dv-|U|^2\right\}+\frac{\nu}{\rho}\left\{\int fv^2_i\sqrt{\mu}dv-U^2_i\right\}$.
Therefore, by the same argument as in (2) above, we have
\begin{eqnarray*}
\displaystyle|\partial^{\alpha}\mathcal{T}_{ij}|&\leq&
\frac{C_{|\alpha|,\nu}}{\rho^{2|\alpha|}}
\left(\sum_{|\alpha_1|\leq N}\int_{\mathbb{R}^3}|\partial^{\alpha_1}f||v|^2\sqrt{\mu}dv\right)\left(1+\sum_{|\alpha_2|\leq N}|\partial^{\alpha_2}\rho|\right)^{|\alpha|}
+C_{|\alpha|,\nu}\mathcal{E}(t)\cr
&\leq&\displaystyle C_{|\alpha|,\nu}\frac{\sqrt{\mathcal{E}(t)}\big(1+\sqrt{\mathcal{E}(t)}\big)^{|\alpha|}}{\big(1-\mathcal{E}(t)\big)^{2|\alpha|}}
+C_{|\alpha|,\nu}\mathcal{E}(t)\cr
&\leq&C_{|\alpha|,\nu}\sqrt{\mathcal{E}(t)}.
\end{eqnarray*}
\end{proof}
%
%
%
%
\begin{lemma} \label{MacroscopicField_{theta}}Let $\mathcal{E}(t)$ be sufficiently small. Then, we have positive constants $C>0$ and $C_{\nu}>0$ independent of $\theta$ such that
\begin{eqnarray*}
&&(1)~|\rho_{\theta}(x,t)-1|\leq C\sqrt{\mathcal{E}(t)},\cr
&&(2)~|U_{\theta}(x,t)|\leq C\sqrt{\mathcal{E}(t)},\cr
&&(3)~\big|\mathcal{T}_{\theta ii }(x,t)-1\big|\leq C_{\nu}\sqrt{\mathcal{E}(t)},\quad(i=1,2,3),\cr
&&(4)~\big|\mathcal{T}_{\theta ij}(x,t)\big|\leq C_{\nu}\sqrt{\mathcal{E}(t)},\hspace{0.95cm}(i<j).
\end{eqnarray*}
\end{lemma}
\begin{proof}
(1) By Lemma \ref{ULestimatesofRhoUT} and the definition of $\rho_{\theta}$, we have
\begin{eqnarray*}
|\rho_{\theta}-1|=\theta |\rho-1|\leq \theta\sqrt{\mathcal{E}(t)}\leq \sqrt{\mathcal{E}(t)}.
\end{eqnarray*}
(2) follows directly from (1), Lemma \ref{ULestimatesofRhoUT} and the definition of $U_{\theta}$:
\begin{eqnarray*}
&&U_{\theta}=\frac{\theta}{\rho_{\theta}} \rho U.
\end{eqnarray*}
(3) We divide the case into $i=j$ and $i\neq j$.
When $i=j$, we have from the definition of $G_{\theta}$ that for $i=1,2,3$:
\begin{eqnarray}\label{recall}
\begin{split}
&\hspace{0.6cm}\frac{1-\nu}{3}\left(\frac{\rho_{\theta}(\Theta_{\theta11}+\Theta_{\theta22}+\Theta_{\theta33})+\rho_{\theta}|U_{\theta}|^2}{2}\right)
+\nu\left(\frac{\rho_{\theta}\Theta_{\theta ii}+\rho U^2_{\theta i }}{2}\right)-\frac{\rho_{\theta}}{2}\cr
&\hspace{1.2cm}=\theta\left[\frac{1-\nu}{3}\left(\frac{\rho (\Theta_{11}+\Theta_{22}+\Theta_{33})+\rho|U|^2}{2}\right)+
\nu\left(\frac{\rho\Theta_{ii}+\rho U^2_{i }}{2}\right)-\frac{\rho}{2}\right],
\end{split}
\end{eqnarray}
Summing over $i=1,2,3$, we obtain
\begin{eqnarray}\label{putitback}
\begin{split}
\frac{\rho_{\theta}(\Theta_{\theta11}+\Theta_{\theta22}+\Theta_{\theta33})}{2}
&=\theta\frac{\rho (\Theta_{11}+\Theta_{22}+\Theta_{33})}{2}\cr
&+\frac{\rho_{\theta}|U_{\theta}|^2-\theta\rho|U|^2}{2}+\frac{3}{2}(\rho_{\theta}-\rho\theta).
\end{split}
\end{eqnarray}
We substitute (\ref{putitback}) back to (\ref{recall}) to get
\begin{eqnarray}\label{putitback2}
\begin{split}
\nu\left(\frac{\rho_{\theta}\Theta_{\theta ii}+\rho_{\theta} U^2_{\theta i }}{2}\right)
&=-\frac{1-\nu}{3}
\left[\theta\left(\frac{\rho(\Theta_{11}+\Theta_{22}+\Theta_{33})+|U|^2}{2}\right)\right]
-\frac{1-\nu}{2}(\rho_{\theta}-\theta\rho)\cr
&+\theta
\left[\frac{1-\nu}{3}\left(\frac{\rho(\Theta_{11}+\Theta_{22}+\Theta_{33})+|U|^2}{2}\right)+\nu\left(\frac{\rho\Theta_{ii}+U^2_{i }}{2}\right)\right]\cr
&+\frac{\rho_{\theta}-\theta\rho}{2}\cr
&=\nu\theta\left(\frac{\rho\Theta_{ii}+\rho U_i^2}{2}\right)+\frac{\nu}{2}(\rho_{\theta}-\theta\rho).
\end{split}
\end{eqnarray}
In view of (\ref{putitback}) and (\ref{putitback2}), we see that 
\begin{eqnarray}\label{putitback3}
\begin{split}
\mathcal{T}_{\theta ii}&=\theta\frac{1-\nu}{3}\left\{\frac{\rho (\Theta_{11}+\Theta_{22}+\Theta_{33})}{\rho_{\theta}}\right\}
+\frac{1-\nu}{3}\left\{\frac{\rho_{\theta}|U_{\theta}|^2-\theta\rho|U|^2}{\rho_{\theta}}\right\}\cr
&+(1-\nu)\left(\frac{\rho_{\theta}-\theta\rho}{\rho_{\theta}}\right)
+\nu\theta\left(\frac{\rho\Theta_{ii}+\rho U_i^2}{\rho_{\theta}}\right)+\nu\left(\frac{\rho_{\theta}-\theta\rho}{\rho_{\theta}}\right)\cr
&=\theta\left[\frac{1-\nu}{3}\left\{\frac{\rho (\Theta_{11}+\Theta_{22}+\Theta_{33})}{\rho_{\theta}}\right\}
+\nu\theta\left(\frac{\rho\Theta_{ii}+\rho U_i^2}{\rho_{\theta}}\right)\right]\cr
&+\frac{1-\nu}{3}\left\{\frac{\rho_{\theta}|U_{\theta}|^2-\theta\rho|U|^2}{\rho_{\theta}}\right\}
+\left(\frac{\rho_{\theta}-\theta\rho}{\rho_{\theta}}\right)\cr
&=\frac{\rho}{\rho_{\theta}}\mathcal{T}_{ii}+\frac{1-\nu}{3}\left\{\frac{\rho_{\theta}|U_{\theta}|^2-\theta\rho|U|^2}{\rho_{\theta}}\right\}
+\left(\frac{\rho_{\theta}-\theta\rho}{\rho_{\theta}}\right).
\end{split}
\end{eqnarray}
Therefore, applying Lemma \ref{ULestimatesofRhoUT}, Lemma \ref{ULestimatesofRhoUT1} and the estimate (1) and (2) above, we find that
\begin{eqnarray*}
\mathcal{T}_{\theta ii}&\leq&\frac{\theta(1+C_{\nu}\sqrt{\mathcal{E}(t)})^2}{1-\sqrt{\mathcal{E}(t)}}
+\frac{1-\nu}{3}\frac{(1+\sqrt{\mathcal{E}(t)})\mathcal{E}(t)}{1-\sqrt{\mathcal{E}(t)}}
+\frac{1+\sqrt{\mathcal{E}(t)}-\theta(1-\sqrt{\mathcal{E}}(t))}{1-\sqrt{\mathcal{E}(t)}}\cr
&\leq&\frac{1+C\theta\sqrt{\mathcal{E}(t)}+C\sqrt{\mathcal{E}(t)}}{1-\sqrt{\mathcal{E}(t)}}.
\end{eqnarray*}
This leads to
\begin{eqnarray}\label{leadsto1}
\mathcal{T}_{\theta ii}-1\leq \frac{C\theta\sqrt{\mathcal{E}}(t)+C\sqrt{\mathcal{E}(t)}}{1-\sqrt{\mathcal{E}(t)}}\leq C\sqrt{\mathcal{E}(t)}.
\end{eqnarray}
Lower bound estimate for $\mathcal{T}_{\theta ii}$ can be derived analogously as
\begin{eqnarray}\label{gives1}
\mathcal{T}_{\theta ii}-1\geq -C\sqrt{\mathcal{E}}(t).
\end{eqnarray}
Combining (\ref{leadsto1}) and (\ref{gives1}), we obtain
\begin{eqnarray*}
|\mathcal{T}_{\theta ii}-1|\leq C\sqrt{\mathcal{E}}(t).
\end{eqnarray*}
The case for $i\neq j$ is simpler. We first observe from the definition of $G_{\theta ij}$ that
\begin{eqnarray*}
\nu\left(\frac{\rho_{\theta}\Theta_{\theta ij}+\rho U_{\theta i }U_{\theta j}}{2}\right)=
\theta\nu\left(\frac{\rho\Theta_{ij}+\rho U_{i }U_{j}}{2}\right),
\end{eqnarray*}
Hence we have
\begin{eqnarray}\label{46}
\mathcal{T}_{\theta ij}=\frac{\rho}{\rho_{\theta}}\theta (\Theta_{ij}+U_iU_j)-U_{\theta i}U_{\theta j}.
\end{eqnarray}
Then we can proceed similarly to obtain the desired result.
\end{proof}
\begin{lemma}\label{Derivatives} Let $\mathcal{E}(t)$ be sufficiently small. Then we have
\begin{eqnarray*}
&&(1)~|\partial^{\alpha}\rho_{\theta}(x,t)|\leq \sqrt{\mathcal{E}(t)},\cr
&&(2)~|\partial^{\alpha}U_{\theta}(x,t)|\leq C_{|\alpha|}\mathcal{E}(t),\cr
&&(3)~|\partial^{\alpha}\mathcal{T}_{\theta}(x,t)|\leq C_{|\alpha|}\mathcal{E}(t).
\end{eqnarray*}
for some positive constant $C_{|\alpha|}$.
\end{lemma}
\begin{proof}
The proof is almost identical to Lemma \ref{ULestimatesofRhoUT1}. We omit the proof.
\end{proof}
%
%
%
%
%
\begin{lemma}\label{Det_Estimate}Let $\mathcal{E}(t)$ be sufficiently small. Then determinant of the temperature tensor $\mathcal{T}_{\nu}$ satisfies
the following estimates:
\begin{eqnarray*}
&&(1)~ |\partial^{\alpha}\{\det\mathcal{T}_{\nu}\}|,~ |\partial^{\alpha}\{\det\mathcal{T}_{\nu\theta}\}|\leq C\sqrt{\mathcal{E}(t)},\cr
&&(2)~ |\det\mathcal{T}_{\nu}|,~ |\det\mathcal{T}_{\nu\theta}|\geq \frac{1}{2},
\end{eqnarray*}
for a positive constant $C$ independent of $\nu$.
\end{lemma}
\begin{proof}
We recall the explicit formula for $\det{\mathcal{T}}_{\nu}$ derived in the proof of Lemma 2.2: 
\begin{eqnarray*}
\begin{split}
\det\mathcal{T}_{\nu}&=\mathcal{T}_{11}\mathcal{T}_{22}\mathcal{T}_{33}-\mathcal{T}_{23}^2\mathcal{T}_{11}-\mathcal{T}_{31}^2\mathcal{T}_{22}
-\mathcal{T}^2_{12}\mathcal{T}_{33},\cr
\det\mathcal{T}_{\theta}&=\mathcal{T}_{\theta11}\mathcal{T}_{\theta22}\mathcal{T}_{\theta33}-\mathcal{T}_{\theta23}^2\mathcal{T}_{\theta11}
-\mathcal{T}_{\theta31}^2\mathcal{T}_{\theta22}
-\mathcal{T}^2_{\theta12}\mathcal{T}_{\theta33}.
\end{split}
\end{eqnarray*}
Then (1) follow from the direct application of the estimates on the derivatives of the macroscopic fields in the preceding lemmas.
To prove (2), we recall from Lemma \ref{ULestimatesofRhoUT1} and Lemma \ref{MacroscopicField_{theta}} that
\begin{eqnarray*}
\mathcal{T}_{ii}=1+o(\mathcal{E}(t))~(i=1,2,3),~\mathcal{T}_{ij}=o~(\mathcal{E}(t)) ~(i\neq j).
\end{eqnarray*}
which leads to
\begin{eqnarray*}
\det\mathcal{T}_{\nu},~\det\mathcal{T}_{\theta\nu}&=&\left\{1+o(\mathcal{E}(t))\right\}^3-1+3\left\{o(\mathcal{E}(t))\right\}^2\left\{1+o(\mathcal{E}(t))\right\}\cr
&\geq&1-o(\mathcal{E}(t))\cr
&\geq&\frac{1}{2}
\end{eqnarray*}
for sufficiently small $\mathcal{E}(t)$.
\end{proof}
\subsection{Uniform estimate on the temperature tensor}
Recall that the nonlinear perturbation $\Gamma(f)$ contains inverse of the temperature tensor $\mathcal{T}^{-1}_{\nu\theta}$:
\begin{eqnarray*}
Q^{\mathcal{M}}_{ij}
=\frac{1}{\sqrt{\mu}}\sum_{i,j}\frac{P^{\mathcal{M}}_{ij}(\rho,v-U_{\theta},U_{\theta},\mathcal{T}^{-1}_{\theta},\nu)}{R^{\mathcal{M}}_{ij}(\rho_{\theta},\det\mathcal{T}_{\theta})}
\exp\left(-\frac{1}{2}(v-U_{\theta})^{\top}\mathcal{T}^{-1}_{\theta}(v-U_{\theta}))\right).
\end{eqnarray*}
Now, since $\mathcal{T}_{\nu}$ (and $\mathcal{T}_{\theta}$) contains $\nu$,
rough estimates of its inverse may involve factors inversely proportional to $\nu$ in it, which make it impossible to derive estimates uniform around $\nu=0$.
This is a serious problem considering that the $\nu=0$ corresponds to the classical BGK model.
In what follows, we will carefully investigate the temperature tensor $\mathcal{T}_{\nu}$ and show that the seemingly problematic $1/\nu$ factor
actually does not cause any harm. The key observation is that $\mathcal{T}_{\nu}$ is essentially equivalent to the temperature $T$
under our assumptions on $\nu$.
%
%
%
%
\begin{proposition}\label{Equivalence T}
Let $-1/2<\nu<1$. Define constant $C_{\nu1}$ and $C_{\nu2}$ by
\begin{eqnarray*}
C_{\nu1}=\min\{1-\nu, 1+2\nu\},\quad C_{\nu2}=\max\{1-\nu, 1+2\nu\}.
\end{eqnarray*}
 Then the temperature tensor is comparable to the temperature in the following sense:
\begin{eqnarray*}
C_{\nu1}T(x,t)Id\leq\mathcal{T}_{\nu}(x,t)\leq C_{\nu2}T(x,t)Id.
\end{eqnarray*}
Furthermore, if $\mathcal{E}(f(t))$ be sufficiently small, then
$\mathcal{T}_{\nu}$ is invertible and
\begin{eqnarray*}
\frac{C^{-1}_{\nu2}}{T(x,t)}Id\leq\mathcal{T}^{-1}_{\nu}(x,t)\leq \frac{C^{-1}_{\nu1}}{T(x,t)}Id.
\end{eqnarray*}
\end{proposition}
\begin{proof}
(1) We first observe from the definition of $\mathcal{T}_{\nu}$ that
\begin{eqnarray*}
\rho\mathcal{T}_{\nu}&=&\left(
\begin{array}{ccc}
(1-\nu) \rho T+\nu\rho \Theta_{11}&\nu\rho \Theta_{12}&\nu\rho \Theta_{13}\cr
\nu\rho \Theta_{21}&(1-\nu)T+\nu\rho \Theta_{22}&\nu\rho \Theta_{23}\cr
\nu\Theta_{31}&\nu\rho \Theta_{32}&(1-\nu) T+\nu\rho \Theta_{33}
\end{array}
\right)\cr
&=&(1-\nu) \rho T Id+\nu\rho \Theta\cr
&=&\frac{(1-\nu)}{3}\int_{\mathbb{R}^3}F(x,v,t)|v-U|^2dv  Id+\nu\int_{\mathbb{R}^3}F(x,v,t)(v-U)\otimes (v-U)dv.
\end{eqnarray*}
Then a direct computation using
\[
k^{\top}\{(v-U)\otimes(v-U)\}k=\{(v-U)\cdot k\}^2
\]
shows that for any $k$ in $\mathbb{R}^3$
\begin{eqnarray*}
k^{\top}\{\rho \mathcal{T}_{\nu}\}k=\frac{(1-\nu)}{3}\left\{\int_{\mathbb{R}^3} F(x,v,t)|v-U|^2dv\right\}|k|^2+\nu\int_{\mathbb{R}^3} F(x,v,t)\big\{(v-U)\cdot k\big\}^2 dv.
\end{eqnarray*}
We split the estimate into the following two cases.
When $0\leq \nu<1$, we have
\begin{eqnarray*}
k^{\top}\{\rho \mathcal{T}_{\nu}\}k\geq\frac{(1-\nu)}{3}|k|^2\int_{\mathbb{R}^3} F(x,v,t)|v-U|^2dv.
\end{eqnarray*}
In the case $-\frac{1}{2}\leq \nu<0$, we apply Cauchy-Schwartz inequality to the second term to get
\begin{eqnarray*}
k^{\top} \{\rho \mathcal{T}_{\nu}\}k&\geq&\frac{(1-\nu)}{3}\left\{\int_{\mathbb{R}^3} F(x,v,t)|v-U|^2dv\right\}|k|^2+\nu\int_{\mathbb{R}^3} F(x,v,t)|v-U|^2 |k|^2 dv\cr
&=&\frac{(1+2\nu)}{3}|k|^2\int_{\mathbb{R}^3} F(x,v,t)|v-U|^2dv.
\end{eqnarray*}
Therefore, we have
\begin{eqnarray}\label{TTT}
k^{\top} \{\rho \mathcal{T}_{\nu}\}k\geq\frac{1}{3}\min\{1-\nu,1+2\nu\}|k|^2\int_{\mathbb{R}^3} F(x,v,t)|v-U|^2dv,
\end{eqnarray}
or equivalently,
\begin{eqnarray}\label{TT}
k^{\top}\mathcal{T}_{\nu}k\geq\min\{1-\nu,1+2\nu\}|k|^2 T ,
\end{eqnarray}
We then apply Lemma \ref{ULestimatesofRhoUT} to compute 
\begin{eqnarray}\label{Tlower}
\begin{split}
T(x,t)&=\frac{1}{\rho}\int_{\mathbb{R}^3}F(x,v,t)|v-U|^2dv\cr
&=\frac{1}{\rho}\left\{\int_{\mathbb{R}^3} F(x,v,t)|v|^2dv-\rho |U|^2\right\}\cr
&=\frac{1}{1-\mathcal{E}(t)}\left\{\int_{\mathbb{R}^3}\left\{ \mu+\sqrt{\mu}f\right\}|v|^2v-C\mathcal{E}(t)\right\}\cr
&\geq\frac{1}{1-\mathcal{E}(t)}\left\{\int_{\mathbb{R}^3} \mu|v|^2dv-\|f\|_{L^{\infty}_{x,v}}\int_{\mathbb{R}^3}\sqrt{\mu}|v|^2dv-C\mathcal{E}(t)\right\}\cr
&\geq\frac{3-C\sqrt{\mathcal{E}(t)}}{1-\mathcal{E}(t)}\cr
&\geq3-C\sqrt{\mathcal{E}(t)}
\end{split}
\end{eqnarray}
for some generic constant $C$.
From (\ref{TT}) and (\ref{Tlower}), we conclude that for any fixed $-1/2<\nu<1$ and for sufficiently small $\mathcal{E}(t)$, $\mathcal{T}_{\nu}$ is
invertible and
\begin{eqnarray*}
\mathcal{T}^{-1}_{\nu}\leq \frac{1}{\min\{1-\nu,1+2\nu\}}T^{-1}Id.
\end{eqnarray*}
The proof for the upper bound is similar.
\end{proof}
%
%
%
%
\begin{lemma}\label{UniformEst}
Let $-1/2<\nu<1$. Suppose $\mathcal{E}(f(t))$ be sufficiently small. Then there exists a positive constant $C_{\nu}<\infty$ such that
\begin{eqnarray*}
~X^{\top}\{\mathcal{T}^{-1}_{\nu}\}Y\leq C_{\nu}\big\{\|X\|^2+\|Y\|^2\},
\end{eqnarray*}
for $X$, $Y$ in $\mathbb{R}^3$.
\end{lemma}
\begin{proof}
By Proposition \ref{Equivalence T}, $\mathcal{T}_{\nu}$ is invertible under the assumption of the lemma.
Moreover, Since $\mathcal{T}_{\nu}$ is symmetric, $\mathcal{T}^{-1}_{\nu}$ also is symmetric.
Therefore, we can compute
\begin{eqnarray*}
\left|X^{\top}\mathcal{T}_{\nu}^{-1}Y\right|&=&\frac{1}{2}\left|(X+Y)^{\top}\mathcal{T}_{\nu}^{-1}(X+Y)-X^{\top}\mathcal{T}_{\nu}^{-1}X-Y^{\top}\mathcal{T}_{\nu}^{-1}Y\right|\cr
&\leq&\frac{1}{2}\left|(X+Y)^{\top}\mathcal{T}_{\nu}^{-1}(X+Y)\right|+\frac{1}{2}\left|X^{\top}\mathcal{T}_{\nu}^{-1}X\right|+\frac{1}{2}\left|Y^{\top}\mathcal{T}_{\nu}^{-1}Y\right|\cr
&\leq&\frac{C}{\min\left\{1-\nu,1+2\nu\right\}}\big\{\|X\|^2+\|Y\|^2\big\}.
\end{eqnarray*}
for any two vectors $X$ and $Y$ in $\mathbb{R}^3$.
\end{proof}
Similar result holds for $\mathcal{T}_{\theta}$:
\begin{lemma}\label{UniformEstTheta}
Let $-1/2<\nu<1$. Suppose $\mathcal{E}(f(t))$ is sufficiently small. Then $\mathcal{T}_{\theta}$ is invertible, and
there exists a positive constant $C_{\nu}<\infty$ such that
\begin{eqnarray*}
~X^{\top}\{\mathcal{T}^{-1}_{\theta}\}Y\leq C_{\nu}\big\{\|X\|^2+\|Y\|^2\}.
\end{eqnarray*}
\end{lemma}
\begin{proof}
In view of (\ref{putitback3}) and (\ref{46}), we can write $\rho\mathcal{T}_{\nu}$ as
\begin{eqnarray*}
\rho_{\theta}\mathcal{T}_{\theta}&=&\theta\rho\mathcal{T}+\left\{\frac{1-\nu}{3}(\rho_{\theta}|U_{\theta}|^2-\theta\rho|U|^2)+\rho_{\theta}-\theta\rho\right\}Id\cr
&+&\nu\theta(\rho\ U\otimes U-\rho D)-\nu(\rho_{\theta}U_{\theta}\otimes U_{\theta}-\rho_{\theta}D_{\theta}),
\end{eqnarray*}
so that
\begin{eqnarray*}
k^{\top}\left\{\rho_{\theta}\mathcal{T}_{\theta}\right\}k&=&\theta k\left\{\rho\mathcal{T}\right\}k
+(\rho_{\theta}-\theta\rho)|k|^2
+\frac{1-\nu}{3}\left(\rho_{\theta}|U_{\theta}|^2-\theta\rho|U|^2\right)|k|^2\cr
&+&\nu\theta\left\{\rho(U\cdot k)^2-\rho k^{\top}Dk\right\}-\nu\left\{\rho_{\theta}(U_{\theta}\cdot k)^2-\rho_{\theta}k^{\top}D_{\theta}k\right\},
\end{eqnarray*}
for $k\in \mathbb{R}^3$.
$D$ and $D_{\theta}$ denote the diagonal matrix with diagonal elements $U^2_1,U^2_2,U^2_3$ and $U^2_{\theta1},U^2_{\theta2},U^2_{\theta3}$
respectively:
\begin{eqnarray*}
D=\left(\begin{array}{ccc}U^2_{1}&0&0\cr0&U^2_{2}&0\cr0&0&U^2_{3}\end{array}\right),\quad
D_{\theta}=\left(\begin{array}{ccc}U^2_{1\theta}&0&0\cr0&U^2_{2\theta}&0\cr0&0&U^2_{3\theta}\end{array}\right).
\end{eqnarray*}
Then, employing Lemma 4.1 and 4.3, we obtain
\begin{eqnarray*}
k^{\top}\{\rho_{\theta}\mathcal{T}_{\theta}\}k&\geq& \frac{1}{2}\min\{1+2\nu,1-\nu\}|k|^2+(1-\theta)|k|^2+O(\mathcal{E}(t))|k|^2\cr
&\geq& \frac{1}{3}\min\{1+2\nu,1-\nu\}|k|^2,\quad k\in \mathbb{R}^3,
\end{eqnarray*}
for sufficiently small $\mathcal{E}(t)$.
By virtue of Lemma \ref{MacroscopicField_{theta}} (1)
\begin{eqnarray*}
k^{\top}\{\mathcal{T}_{\theta}\}k\geq \frac{1}{4}\min\{1+2\nu,1-\nu\}|k|^2,
\end{eqnarray*}
The rest of the proof is similar to the proof of Lemma \ref{UniformEst}.
\end{proof}
%
%
%
%
\begin{lemma}\label{UniformEst2}
Let $-1/2<\nu<1$. Suppose $\mathcal{E}(f(t))$ is sufficiently small. Then there exists a positive constant $C_{\nu,\alpha}<\infty$ such that
\begin{eqnarray*}
(1)~X^{\top}\{\partial^{\alpha}\mathcal{T}^{-1}_{\nu}\}Y\leq C_{\nu,\alpha}\big\{\|X\|^2+\|Y\|^2\},\cr
(2)~X^{\top}\{\partial^{\alpha}\mathcal{T}^{-1}_{\nu\theta}\}Y\leq C_{\nu,\alpha}\big\{\|X\|^2+\|Y\|^2\},
\end{eqnarray*}
for $X$, $Y$ in $\mathbb{R}^3$.
\end{lemma}
\begin{proof}
We have proved in Lemma \ref{UniformEst} that $\mathcal{T}_{\nu}$ is strictly positive definite for $-1/2<\nu<1$ when $\mathcal{E}(t)$ is
sufficiently small. Therefore,
$\mathcal{T}_{\nu}$ is invertible. Now, applying $\partial$ to $\mathcal{T}_{\nu}\mathcal{T}^{-1}_{\nu}=I$, we see that
$\partial\mathcal{T}_{\nu}\left\{\mathcal{T}^{-1}_{\nu}\right\}+\mathcal{T}_{\nu}\partial\left\{\mathcal{T}^{-1}_{\nu}\right\}=0$, and thus,
\begin{eqnarray}\label{dugod}
\partial\{\mathcal{T}^{-1}_{\nu}\}=\mathcal{T}^{-1}_{\nu}\left\{\partial\mathcal{T}_{\nu}\right\}\mathcal{T}^{-1}_{\nu}.
\end{eqnarray}
Then the case $|\alpha|=1$ follows directly from this identity and Lemma \ref{UniformEst} and Lemma \ref{Derivatives}.
For general case, we recall
\begin{eqnarray}\label{dugod2}
\partial^{\alpha}\mathcal{T}_{\nu}=\sum_{|\beta|+|\gamma|=|\alpha|} \partial^{\beta}\mathcal{T}_{\nu}\partial^{\gamma}\left\{\mathcal{T}^{-1}_{\nu}\right\}
\end{eqnarray}
and use the induction argument. The proof for $\mathcal{T}_{\theta}$ is almost identical. We omit it.
\end{proof}
\subsection{Local existence}
We first estimate the nonlinear term $\Gamma(f)$.
Note that, in contrast to the Boltzmann equation, we need to use the estimates on the macroscopic fields established in the previous section
to control $\Gamma(f)$ in the energy norm.
%
%
%
%
%
%
%
%
\begin{lemma}\label{GammaEstimate} The bilinear perturbation $\Gamma$ satisfies the following estimates:
\begin{eqnarray*}
&&(1)~ \left|~\int \partial^{\alpha}_{\beta}\Gamma(f)gdv\right|\leq C\sum_{\substack{|\alpha_1|+|\alpha_2|\\\leq|\alpha|}}
\|\partial^{\alpha_1} f\|_{L^2_{x,v}}\|\partial^{\alpha_2} f\|_{L^2_v}\|h\|_{L^2_v}\cr
&&\hspace{3.4cm}+C\!\!\!\!\!\sum_{\substack{|\alpha_1|+|\alpha_2|\leq|\alpha|\\|\beta_2|\leq |\beta|}}
\|\partial^{\alpha_1}f\|_{L^2_{x,v}}\|\partial^{\alpha_2}_{\beta_2}f\|_{L^2_{v}}\|\partial^{\alpha_3}f\|_{L^2_v}\|h\|_{L^2_v},\cr
&&\hspace{3.4cm}+C\!\!\!\!\!\sum_{\substack{|\alpha_1|+|\alpha_2|+|\alpha_3|\\\leq|\alpha|}}\|\partial^{\alpha_1}f\|_{L^2_{x,v}}\|\partial^{\alpha_2}f\|_{L^2_{v}}
\|\partial^{\alpha_3}f\|_{L^2_v}\|h\|_{L^2_v},\cr
&&(2)~ \Big|\int\Gamma_{1,2}(f,g)fdv\Big|+\Big|\int\Gamma_{1,2}(g,f)fdv\Big|
\leq C\sup_x\|g\|_{L^2_{x,v}}\|f\|^2_{L^2_{x,v}}.\cr
&& \hspace{0.65cm}\Big|\int\Gamma_{3}(f,g,h)fdv\Big|+\Big|\int\Gamma_3(g,f,h)fdv\Big|+\Big|\int\Gamma_3(g,h,f)fdv\Big|\cr
&&\hspace{1.5cm}\leq C\sup_x\|g\|_{L^2_{v}}\sup_x\|h\|_{L^2_v}\|f\|^2_{L^2_{x,v}}.\cr
&&(3)~ \left\|\Gamma_{1,2}(f,g)h+\Gamma_{1,2}(g,h)h\right\|_{L^2_{x,v}}\leq C\sup_{x,v}|h|\sup_x\|f\|_{L^2_v}\|g\|_{L^2_{x,v}},\cr
&&\hspace{0.7cm} \left\|\Gamma_3(f,g,h)r+\Gamma_3(g,f,h)r+\Gamma_3(g,h,f)r\right\|_{L^2_{x,v}}\cr
&&\hspace{1.43cm}\leq C\sup_{x,v}|h|\sup_x\|f\|_{L^2_v}\sup_x\|g\|_{L^2_v}\|h\|_{L^2_{x,v}}.
\end{eqnarray*}
\end{lemma}
\begin{proof}
Recall that the $\Gamma$ consists of $\Gamma_1$, $\Gamma_2$ and $\Gamma_3$. We prove this lemma only for $\Gamma_2$, because the proof
for the remaining parts are similar. Utilizing macroscopic estimates established in the previous section, we find
that there exists a polynomial $P_{\alpha,\beta}$, which is generically defined, such that
\begin{eqnarray*}
&&\left|\partial^{\alpha}_{\beta}\mathcal{M}_{\nu}\left(\rho_{\theta},U_{\theta},\mathcal{T}_{\nu\theta}\right)\right|\cr
&&\qquad=C_{\alpha,\beta}\left|P_{\alpha,\beta}(\nu,\partial\rho_{\theta},\partial U_{\theta},\partial\mathcal{T}_{\nu\theta})\right|\exp\left(-\frac{1}{2}(v-U_{\theta})^{\top}\mathcal{T}^{-1}_{\nu\theta}(v-U_{\theta})\right)\cr
&&\qquad\leq C_{\alpha,\beta}P(v)\exp\left(-\big\{1-o\big(\mathcal{E}\big(f(t)\big)\big)\big\}\frac{|v|^2}{2}+o\big(\mathcal{E}\big(f(t)\big)\big) \right)\cr
&&\qquad\leq C_{\alpha,\beta,\varepsilon}\exp\left(-\big\{1-o(\mathcal{E}(t))\big\}\left(\frac{1}{2}-\varepsilon\right)|v|^2+o\big(\mathcal{E}\big(f(t)\big)\big)\right),
\end{eqnarray*}
where $\partial$ denotes any of $\partial^{m}_{n}$ such that $m\leq|\alpha|$ and $n\leq|\beta|$.
Therefore, there exists a positive number $a$ depending on $\alpha$, $\beta$ and $\nu$ such that
\begin{eqnarray}\label{Q_Estimate}
\frac{1}{\sqrt{\mu}}\left|\partial^{\alpha}_{\beta}Q^{\mathcal{M}}\right|\leq C_{\alpha,\beta}\exp\left(-a|v|^2\right)
\end{eqnarray}
(1) By (\ref{Q_Estimate}) and H\"{o}lder inequality, we see
\begin{eqnarray*}
\int_{\mathbb{R}^3}|\partial^{\alpha}_{\beta}\Gamma_2(f)g|dv&\leq& \sum_{\substack{|\alpha_1|+|\alpha_2|+|\alpha_3|\cr=|\alpha|}}
\int_{\mathbb{R}^3}\left|\partial^{\alpha}_{\beta}Q^{\mathcal{M}}
\langle \partial^{\alpha_1}f,e_i\rangle\langle \partial^{\alpha_2}f,e_j\rangle gdv\right|\cr
&\leq&C\sum_{\substack{|\alpha_1|+|\alpha_2|+|\alpha_3|\cr=|\alpha_3|}}
\int_{\mathbb{R}^3} \exp\left(-a|v|^2\right)
\langle \partial^{\alpha_1}f,e_i\rangle\langle \partial^{\alpha_2}f,e_j\rangle gdv\cr
&\leq&C\sum_{\substack{|\alpha_1|+|\alpha_2|+|\alpha_3|\cr=|\alpha|}}
\left(\int_{\mathbb{R}^3}\exp\left(-|v|^2\right)gdv\right)
\|\partial^{\alpha_1}f\|_{L^2_v}\|\partial^{\alpha_2}f\|_{L^2_v}\cr
&\leq&C\sum_{\substack{|\alpha_1|+|\alpha_2|+|\alpha_3|\cr=|\alpha|}}
\left\|\exp\left(-a|v|^2\right)\right\|_{L^2v}\|g\|_{L^2v}
\|\partial^{\alpha_1}f\|_{L^2_v}\|\partial^{\alpha_2}f\|_{L^2_v}\cr
&\leq&C\sum_{\substack{|\alpha_1|+|\alpha_2|+|\alpha_3|\cr=|\alpha|}}
\|\partial^{\alpha_1}f\|_{L^2_v}\|\partial^{\alpha_2}f\|_{L^2v}\|g\|_{L^2_v}.
\end{eqnarray*}
(2) can be estimated similarly as
\begin{eqnarray*}
\int \Gamma_2(f,g)fdxdv&\leq&C\int_{\mathbb{R}^3}\|f\|_{L^2_v}\|g\|_{L^2_v}
\left(\int_{\mathbb{R}^3}\exp\left(-a|v|^2\right)fdv\right)dx\cr
&\leq&C\int_{\mathbb{R}^3}\|f\|_{L^2_v}\|g\|_{L^2_v}\|f\|_{L^2_v}dx\cr
&\leq&C\sup_x\|g\|_{L^2_v}\|f\|^2_{L^2_{x,v}}.
\end{eqnarray*}
(3) For $\Phi\in L^2_{x,v}$, we have
\begin{eqnarray*}
\langle\Gamma_2(f,g)r,\Phi\rangle&\leq&C\int_{\mathbb{R}^3}\|f\|_{L^2_v}\|g\|_{L^2_v}\|r\Phi\|_{L^2_v}dx\cr
&\leq&C\sup_{x,v}|r|\int_{\mathbb{R}^3}\|f\|_{L^2_v}\|g\|_{L^2_v}\|\Phi\|_{L^2_v}dx\cr
&\leq&C\sup_{x,v}|r|\left(\int_{\mathbb{R}^3}\|f\|^2_{L^2_v}\|g\|^2_{L^2_v}dx\right)^{\frac{1}{2}}\|\Phi\|_{L^2_{x,v}}\cr
&\leq&C\sup_{x,v}|r|\sup_x\|f\|_{L^2_v}\|g\|_{L^2_{x,v}}\|\Phi\|_{L^2_{x,v}}.
\end{eqnarray*}
Therefore, the duality argument gives
\begin{eqnarray*}
\|\Gamma_2(f,g)r\|_{L^2_{x,v}}\leq C\sup_{x,v}|r|\sup_x\|f\|_{L^2_v}\|g\|_{L^2_{x,v}}.
\end{eqnarray*}
\end{proof}
From the estimates in Lemma \ref{GammaEstimate}, the following local existence theorem can be proved by standard arguments (See, e.g \cite{Guo-VMB,Yun}).
\begin{theorem}\label{localExistence}
Let $\nu$ be a fixed constant such that $-1/2\leq \nu<1$. Let $F_0=g_0+\sqrt{\mu}f_0\geq 0$ and $f_0$ satisfies the conservation laws (\ref{ConservationLawsf}).
Then there exists $M_0>0$, $T_*>0$, such that if $ T_*\leq\frac{M_0}{2}$ and $\mathcal{E}(f_0)<\frac{M_0}{2}$, there is a unique solution $f(x,v,t)$ to the ES-BGK model
(\ref{LBGK}) such that
\begin{enumerate}
\item The high order energy $\mathcal{E}\big(f(t)\big)$ is continuous in $[0,T*)$ and uniformly bounded:
\begin{eqnarray*}
\sup_{0\leq t\leq T_*}\mathcal{E}\big(f(t)\big)\leq  M_0.
\end{eqnarray*}
\item The distribution function remains positive in $[0,T_*)$:
\begin{eqnarray*}
F(x,v,t)=\mu+\sqrt{\mu}f(x,v,t)\geq 0.
\end{eqnarray*}
\item The conservation laws (\ref{ConservationLawsf}) hold for all $[0,T_*]$.
\end{enumerate}
\end{theorem}
\begin{proof}
We consider the following iteration scheme.
\begin{eqnarray}\label{Localscheme}
\partial_tF^{n+1}+v\cdot\nabla_xF^{n+1}=\frac{\rho_n T_n}{1-\nu}\left\{\mathcal{M}_{\nu}(F^n)-F^{n+1}\right\},
\end{eqnarray}
where $\mathcal{M}(F^n)$ is defined by
\begin{eqnarray*}
\mathcal{M}_{\nu}(F^n)=\frac{\rho^n}{\sqrt{\det(2\pi\mathcal{T})}}\exp\left(\frac{1}{2}(v-U^n)\left\{\mathcal{T}^{n}_{\nu }\right\}^{-1}(v-U^n)\right).
\end{eqnarray*}
$\rho^n$, $U^n$ and $\mathcal{T}^n_{\nu}$ denote the local density, bulk velocity and the temperature tensor associated with $F^n=\mu+\sqrt{\mu}f^n$.
With estimates on the nonlinear perturbation in Lemma \ref{GammaEstimate}, it is standard to prove the local existence
(See \cite{Guo-VMB,Yun}).
The only thing to be careful about is whether the temperature tensor $\mathcal{T}^n_{\nu}$ remains strictly positive definite
for each $n$, so that the iteration scheme is well-defined in each step. But this follows directly from Proposition \ref{Equivalence T} and Lemma \ref{UniformEst} - \ref{UniformEst2}.
\end{proof}
%
%
%
%
\section{Global Existence}
Now, having all the necessary estimates at hand, the global existence can be established using standard arguments (See \cite{Guo-VMB,Yun}).
We sketch the proof in this section. First, we need to recover the degeneracy of the
linearized relaxation operator to obtain the full coercivity. For this, we define
\begin{eqnarray*}
&&a(x,t)=\int_{\mathbb{R}^3}f\sqrt{\mu}dv,~ b_i(x,t)=\int_{\mathbb{R}^3}fv_i\sqrt{\mu}dv ~(i=1,2,3),
~ c(x,t)=\int_{\mathbb{R}^3}f|v|^2\sqrt{\mu}dv.
\end{eqnarray*}
We also define a macroscopic projection $P$ as follows:
\begin{eqnarray*}
Pf= a(x,t)\sqrt{\mu}+\sum_{i}b_i(x,t)v_i\sqrt{\mu}+c(x,t)|v|^2\sqrt{\mu}.
\end{eqnarray*}
Note that $P$ is not identical to $P_0$ but equivalent. Since $L_{\nu}\{Pf\}=0$ for  $-1/2<\nu<1$ by Corollary \ref{Kernel L},
we can split the linearized ES-BGK model (\ref{LBGK}) into the macroscopic part and the microscopic part as follows:
\begin{eqnarray*}
\{\partial_t+v\cdot\nabla_x\}\{Pf\}=-\{\partial_t+v\cdot\nabla_x\}\{(I-P)f\}+L\{(I-P)f\}+\Gamma(f).\
\end{eqnarray*}
%
%
%
%
%
We then expand the l.h.s and r.h.s with respect to the following basis $(1\leq i,j\leq 3)$:
\begin{eqnarray}\label{basis}
\big\{
\sqrt{\mu},v_i\sqrt{\mu}, v_iv_j\sqrt{\mu}, v_i^2\sqrt{\mu}, v_i|v|^2\sqrt{\mu}
\big\},
\end{eqnarray}
and compare coefficients on both sides to obtain the following micro-macro system \cite{Guo-VMB}:
\begin{eqnarray}\label{MicroMacro}
\partial_t a&=&\ell_{a}+h_{a},\nonumber\\
\partial_tb_i+\partial_{x_i}a&=&\ell_{abi}+h_{abi},\nonumber\\
\partial_{x_i}b_j+\partial_{x_j}b_i&=&\ell_{ij}+h_{ij}\quad (i\neq j)\\
\partial_{x_i} b_i+\partial_{t}c&=&\ell_{bci}+h_{bci},\nonumber\\
\partial_{x_i}c&=&\ell_{ci}+h_{ci},\nonumber
\end{eqnarray}
for $i,j =1,2,3$.
Then, by carefully studying this system, we find that the macroscopic part can be controlled by the macroscopic part as follows (See \cite{Guo-VMB}):
%
%
%
%
\begin{eqnarray}\label{MicroMacroEstimate}
\sum_{|\alpha|\leq N}\left\{\|\partial^{\alpha}a\|_{L^2_{x}}+\|\partial^{\alpha}b\|_{L^2_{x}}+\|\partial^{\alpha}c\|_{L^2_{x}}\right\}
\leq C \sum_{|\alpha|\leq N-1}\|\partial^{\alpha}(\ell_{\nu}+h_{\nu})\|_{L^2_{x}}.
\end{eqnarray}
We slightly abused the notation on the r.h.s for the simplicity of presentation. 
On the other hand, we can bound $\ell_{\nu}$ and $h_{\nu}$ by the energy norm of $f$ as
%
%
%
%
\begin{eqnarray*}
\sum_{|\alpha|\leq N-1}\|\partial^{\alpha}(\ell^{\nu}+h^{\nu})\|_{L^2_{x}}
\leq C_{\nu}\sum_{|\alpha|\leq N}\|(I-P)\partial^{\alpha}f\|_{L^2_{x,v}}+
C_{\nu}\sqrt{M_0}\sum_{|\alpha|\leq N}\|\partial^{\alpha}f\|_{L^2_{x,v}}.
\end{eqnarray*}
%
%
%
%
Combining this with (\ref{MicroMacroEstimate}), we see that
\begin{eqnarray*}
\sum_{|\alpha|\leq N}\|\partial^{\alpha}Pf\|^2_{L^2_{x,v}}
&\leq&\frac{1}{C}\sum_{|\alpha|\leq N}\left\{\|\partial^{\alpha}a\|^2_{L^2_{x,v}}+\|\partial^{\alpha}b\|^2_{L^2_{x,v}}+\|\partial^{\alpha}c\|^2_{L^2_{x,v}}\right\}\cr
&\leq& C\sum_{|\alpha|\leq N}\|\partial^{\alpha}(I-P)f\|^2_{L^2_{x,v}}
+C\sqrt{M}_0\sum_{|\alpha|\leq N}\|\partial^{\alpha}f\|^2_{L^2_{x,v}},
\end{eqnarray*}
which implies
\begin{eqnarray}\label{microMacroEstimate}
\sum_{|\alpha|\leq N}\|P\partial^{\alpha}f\|^2_{L^2_{x,v}}
\leq C\sum_{|\alpha|\leq N}\|(I-P)\partial^{\alpha}f\|^2_{L^2_{x,v}}.
\end{eqnarray}
Therefore, Proposition \ref{degenerate coercivity} together with (\ref{microMacroEstimate}) and the equivalence of $P_0$ and $P$ imply
the coercivity estimate for $L_{\nu}$:
There exists $\delta_{\nu}=\delta(\nu)>0$ such that
\begin{eqnarray}\label{Coercivity}
\sum_{|\alpha|\leq N}\langle L_{\nu}\partial^{\alpha}f,\partial^{\alpha}f\rangle\leq -\delta_{\nu}\sum_{|\alpha|\leq N}\|\partial^{\alpha}f(t)\|_{L^2_{x,v}},
\end{eqnarray}
when $f$ is sufficiently small in the energy norm.
We are now ready to derive the nonlinear energy estimates which enables us to extend the local solution into the global one.
%
%
%
%
Let $f$ be the smooth local in time solution constructed in Theorem \ref{localExistence}.
Taking $\partial^{\alpha}$ on both sides of (\ref{LBGK}), we obtain
\begin{eqnarray*}
\partial_t \partial^{\alpha}f+v\cdot\nabla_x \partial^{\alpha}f=L\partial^{\alpha}f+\partial^{\alpha}\Gamma(f).
\end{eqnarray*}
We then take inner product with $\partial^{\alpha}f$
\begin{eqnarray*}
\frac{d}{dt}\|\partial^{\alpha}f\|^2_{L^2_{x,v}}\leq \langle L\partial^{\alpha}f,\partial^{\alpha}f\rangle_{L^2_{x,v}}+\langle\partial^{\alpha}\Gamma(f),\partial^{\alpha}f\rangle_{L^2_{x,v}},
\end{eqnarray*}
and apply the coercivity estimates (\ref{Coercivity}) together with the nonlinear estimates in Lemma \ref{GammaEstimate} to derive
\begin{eqnarray*}
E^{\alpha}_0:\quad\frac{1}{2}\frac{d}{dt}\|\partial^{\alpha}f\|^2_{L^2_{x,v}}\!\!+\delta_{\nu}\sum_{|\alpha|\leq N}\|\partial^{\alpha}f\|^2_{L^2_{x,v}}\leq
C\sqrt{\mathcal{E}(f(t))}\mathcal{D}(f(t)),
\end{eqnarray*}
where $\mathcal{D}(f(t))$ denotes
\[
\mathcal{D}(f(t))=\sum_{|\alpha|+|\beta|\leq N}\|\partial^{\alpha}_{\beta}f(t)\|^2_{L^2_{x,v}}.
\]
We now turn to the general case involving the derivatives in the velocity variables. Applying $\partial^{\alpha}_{\beta}$ to (\ref{LBGK}), we get
\begin{eqnarray*}
\big\{\partial_t+v\cdot\nabla_x +\nu_0\big\}\partial^{\alpha}_{\beta}f
=\partial_{\beta_1}v\cdot\nabla_x\partial^{\alpha}_{\beta-\beta_1}f
+\partial_{\beta}P\partial^{\alpha}f+\partial^{\alpha}_{\beta}\Gamma(f,f).
\end{eqnarray*}
We multiply $\partial^{\alpha}_{\beta}f$, integrate over $\mathbb{R}^3_x\times \mathbb{R}^3_v$ and apply H\"{o}lder inequality with Lemma \ref{GammaEstimate} to see
\begin{eqnarray*}
\begin{split}
E^{\alpha}_{\beta}:\quad\frac{1}{2}\frac{d}{dt}\|\partial^{\alpha}_{\beta}f\|^2_{L^2_{x,v}}+\nu_0\|\partial^{\alpha}_{\beta}f\|^2_{L^2_{x,v}}
&\leq C\sum_{i}\|\partial^{\alpha+e_i}_{\beta-e_i}f\|_{L^2_{x,v}}\|\partial^{\alpha}_{\beta}f\|_{L^2_{x,v}}\cr
&+C\|\partial^{\alpha}f\|_{L^2_{x,v}}\|\partial^{\alpha}_{\beta}f\|_{L^2_{x,v}}+C\sqrt{\mathcal{E}f((t))}\mathcal{D}(f(t)).
\end{split}
\end{eqnarray*}
Then, we split the first two terms in the r.h.s using Young's inequality and gather relevant terms together to obtain
\begin{eqnarray*}
E^{\alpha}_{\beta}:\quad\frac{1}{2}\frac{d}{dt}\|\partial^{\alpha}_{\beta}f\|^2_{L^2_{x,v}}+\frac{\nu_0}{2}\|\partial^{\alpha}_{\beta}f\|^2_{L^2_{x,v}}
&\leq&C_{\varepsilon}\sum_{i}\|\partial^{\alpha+e_i}_{\beta-e_i}f\|^2_{L^2_{x,v}}+C_{\varepsilon}\|\partial^{\alpha}f\|^2_{L^2_{x,v}}\cr
&+&C\mathcal{E}(f(t))\mathcal{D}(f(t)).
\end{eqnarray*}
Now, we observe that the r.h.s of $\sum_{|\beta|=m+1}E^{\alpha}_{\beta}$ can be controlled by the good terms of
\begin{eqnarray*}
C_{m}\sum_{|\beta|=m}E^{\alpha}_{\beta}+C_m\sum_{\alpha}E^{\alpha}
\end{eqnarray*}
if $C_m$ is sufficiently large. By good terms, we mean the production terms on the l.h.s.
Therefore, we can find constants $C_m$ and $\delta_m$ inductively such that
\begin{eqnarray*}
\sum_{\substack{|\alpha|+|\beta|\leq N,\cr|\beta|\leq m}}\left\{C_m\frac{d}{dt}\|\partial^{\alpha}_{\beta}f\|^2_{L^2_{x,v}}
+\delta_m\|\partial^{\alpha}_{\beta}f\|^2_{L^2_{x,v}}\right\}\leq C_N\sqrt{\mathcal{E}(f(t))}\mathcal{D}(f(t)).
\end{eqnarray*}
From this energy estimate, the existence of global solutions follows from the standard continuity argument.
Remaining part of the Theorem 1.1 can be established in the exactly same manner as in the classical BGK case
\cite{Yun}.
This completes the proof.
\begin{section}{Acknowledgement}
The author would like to thank Prof. Yan Guo and Prof. Kazuo Aoki and Prof. Tai-Ping Liu for fruitful discussions.
This research was supported by Basic Science Research Program through
the National Research Foundation of Korea (NRF) funded by the Ministry of Science, ICT $\&$
Future Planning (NRF-2014R1A1A1006432)
\end{section}

\end{document}